\documentclass{article}
\usepackage{blindtext}
\usepackage[a4paper, total={6in, 8in}, includehead, headheight=5mm, margin=1.35in]{geometry}
\usepackage[utf8]{inputenc}
\usepackage{fancyhdr}
\usepackage{amsfonts}\usepackage{bbm}
\usepackage{amsthm}
\usepackage{amsmath}
\usepackage{enumerate}
\usepackage{mathtools}
\usepackage{tikz-cd}
\usepackage{thmtools}
\usepackage{graphicx}
\usepackage[ddmmyy]{datetime}
\usepackage{stmaryrd}
\usepackage{titlesec}
\usepackage{mathabx}
\usepackage{mathrsfs}
\usepackage{enumitem}
\usepackage{scalerel}[2016/12/29]
\usepackage[T1]{fontenc}
\usepackage{tgtermes}
\usepackage[colorlinks,linkcolor=black]{hyperref}
\usepackage[nameinlink]{cleveref}

\newcommand{\ten}{\otimes}
\newcommand{\HH}{\operatorname{HH}}

\newcommand{\bimods}{\textrm{-bimod}}
\newcommand{\bigmods}{\textrm{-bimod}}
\newcommand{\twisten}{\otimes^t\hspace{-1pt}}
\newcommand{\homo}{\operatorname{Hom}}
\newcommand{\bv}{\operatorname{BV}}
\newcommand{\awt}{\operatorname{AW}^t}
\newcommand{\ezt}{\operatorname{EZ}^t}
\newcommand{\pr}{\operatorname{pr}}
\newcommand{\inc}{\operatorname{in}}
\newcommand{\elr}{[r_1 | \cdots | r_n]}
\newcommand{\els}{[s_{n+1} | \cdots |s_{n+m}]}
\newcommand{\bezt}{\underline{\operatorname{EZ}}^t}
\newcommand{\bawt}{\underline{\operatorname{AW}}^t}
\newcommand{\sign}{\operatorname{sign}}

\theoremstyle{plain}
\newtheorem{theorem}{Theorem}
\newtheorem*{theorem*}{Theorem}

\newtheorem{lemma}[theorem]{Lemma}
\newtheorem{corollary}[theorem]{Corollary}
\newtheorem*{corollary*}{Corollary}
\newtheorem{examples}[theorem]{Example}

\theoremstyle{definition}
\newtheorem{definitionproposition}[theorem]{Definition/Proposition}
\newtheorem{notation}[theorem]{Notation}
\newtheorem*{definition*}{Definition}
\newtheorem{remark}[theorem]{Remark}

\newtheorem{defn}[theorem]{Definition}

\makeatletter
\renewcommand{\env@cases}[1][@{}l@{\quad}l@{}]{%
  \let\@ifnextchar\new@ifnextchar
  \left\lbrace
  \def\arraystretch{1.2}%
  \array{#1}%
}
\makeatother

\newlist{pfparts}{enumerate}{1}
\setlist[pfparts,1]{%
  label=\arabic.,
  font=\normalfont\textbf,
  itemindent=2pt,
  wide,
  itemsep=0pt,topsep=2pt,
  labelsep=0.75ex
}
\makeatletter
\def\namedlabel#1#2{\begingroup
    #2%
    \def\@currentlabel{#2}%
    \phantomsection\label{#1}\endgroup 
}
\makeatother

\newtheorem{proposition}[theorem]{Proposition}

\declaretheoremstyle[
headfont=\color{black}\normalfont\bfseries,
notefont = \color{red}\normalfont,
bodyfont=\color{black}\normalfont\itshape,
]{coloredDefinition}

\declaretheoremstyle[
headfont=\color{black}\normalfont\bfseries,
notefont = \color{teal}\normalfont,
bodyfont=\color{black}\normalfont\itshape,
]{coloredExample}

\renewenvironment{abstract}
{\begin{quote}
\noindent\par{\scshape\abstractname.}}
{\medskip\noindent
\end{quote}
}

\title{\Large BV Structure on the Hochschild Cohomology\\ of Twisted Tensor Products}
\author{\textsc{Matthew Antrobus}}
\date{}

\makeatletter
\let\inserttitle\@title
\makeatother

\titleformat{\section}
{\normalfont\large\scshape}{\thesection.}{1em}{}
\titleformat{\subsection}
  {\normalfont\normalsize\scshape}{\thesubsection.}{1em}{}

\begin{document}
\maketitle
\begin{abstract} 
Given  $k$-algebras $R$ and $S$, graded by Abelian groups $A$ and $B$ respectively, and a bicharacter $t:A \otimes B \rightarrow k^\times$, one may form the twisted tensor product algebra $R\twisten S$. The Hochschild Cohomology of such algebras has been extensively studied, for its relation to Happel's conjecture and for use in calculations. Should each of $R$ and $S$ be Frobenius, admitting a semisimple Nakayama Automorphism, then each Hochschild Cohomology admits the structure of a $\bv$-algebra. We describe, under a grading condition on their Frobenius structures, the $\bv$-structure on the Hochschild Cohomology of their twisted tensor product. In particular, this greatly simplifies computations of the $\bv$-structure of the Hochschild Cohomology of Quantum Complete Intersections. Our result specialises to give new results in the untwisted theory.
\end{abstract}
\vspace{-10pt}
\tableofcontents
\newpage
\section{Introduction}
The Hochschild Cohomology of a tensor product of algebras twisted by a bicharacter has been studied extensively in different guises. The Hochschild cohomology of quantum complete intersections (which are twisted tensor products of algebras of the form $k[x]/x^n$) has been studied by various authors (\cite{erdmann}, \cite{bergh}, \cite{oppermann}). In particular, they were studied in \cite{buchweitz}, to answer a question of Happel, regarding finiteness of Hochschild Cohomology.

Generalising this work and as a step towards understanding the Hochschild Cohomology of more general twisted tensor products, Briggs and Witherspoon (\cite{BriggsWitherspoon}) demonstrated a decomposition for the Hochschild cohomology of a bicharacter twisted tensor product. They further gave formulas for both the product and Gerstenhaber bracket in terms of the Hochschild Cohomologies of the original algebras, with twisted coefficients. 

In this paper, we extend their work to describe the $\bv$-structure on the Hochschild cohomology of a twisted tensor product of Frobenius algebras (under mild conditions). In particular, we describe the $\bv$-structure on the Hochschild cohomology of the twisted tensor product in terms of twisted $\bv$-structures on the original cohomologies. A $\bv$-structure on Hochschild cohomology exists for certain classes of algebras, and manifests as a degree minus one operator with square zero; it is particularly useful in that the Gerstenhaber Bracket measures the failure of the $\bv$-operator to be a derivation. Thus calculating the $\bv$-structure and product structure gives the Gerstenhaber bracket. 

This additionally extends the paper of Hou and Wu, (\cite{hou}) wherein the $\bv$-structures of a restricted class of quantum complete intersections are calculated. Our result further extends the literature of usual (non-twisted) tensor products in the following way. It was initially proved by T. Tradler in \cite{Tradler} that the Hochschild Cohomology of a symmetric algebra (i.e, a Frobenius Algebra with symmetric inner product) is a $\bv$-algebra; it was subsequently proved by Le and Zhou \cite{Zhou} that this $\bv$-structure is respected by taking a tensor product of algebras. It was shown slightly later by Lambre, Zhou and Zimmerman (\cite{Lambre}) that we can generalise the result of Tradler. The failure of the inner product on a Frobenius algebra to be symmetric is measured (in some sense), by the Nakayama automorphism; should this automorphism be semisimple,  we also get $\bv$-structure. To the authors' knowledge, the proof of \cite{Zhou} has not been generalised to the (non-twisted), non-symmetric case, though our result specialises to this. 

Our main result is as follows, stated in terms of the decomposition of Hochschild Cohomology from \cite{BriggsWitherspoon}, which is explained later. Let $R$, $S$ be algebras over a field $k$, graded by Abelian groups $A$ and $B$ respectively, with a bicharacter $t:A \otimes B \rightarrow k^\times$. Suppose that there exist $\sigma_R \in A$, $\sigma_S \in B$ so that the morphisms
$$\begin{aligned}{\langle-,-\rangle_R}: R \otimes R \rightarrow k[\sigma_R] \hspace{20pt}{\langle-,-\rangle_S}: S \otimes S \rightarrow k[\sigma_S] \end{aligned}$$
are graded, and that the Nakayama Automorphisms $\nu_R: R \rightarrow R$ and $\nu_S: S \rightarrow S$ are semisimple. Then $R\twisten S$ admits a semisimple Nakayama automorphism, and we have the following theorem.
\begin{theorem*}  The $\bv$-operator on the Hochschild Cohomology $\HH^*(R \twisten S, R \twisten S)$ can be expressed in the following way. For $f\in C^*(R,R_{\hat{b}})^a$ and $g \in C^*(S, \prescript{}{\hat{a}}{S})^b$, we have
$${\Delta}(f\boxtimes g) =  (-1)^{mn}(\Delta_c(f) \boxtimes g +(-1)^{n} f \boxtimes \prescript{}{d}{\Delta(g)})$$
where $$\Delta_b: C^*(R,R_{\hat{b}}) \rightarrow C^{*-1}(R,R_{\hat{b}})  \hspace{20pt}\Delta_a: C^{*-1}(S, \prescript{}{\hat{a}}{S})\rightarrow C^*(S, \prescript{}{\hat{a}}{S}) $$ are twisted $\bv$-operators.
\end{theorem*}
This of course specialises, by taking a trivial twisting.
\begin{corollary*} Given Frobenius algebras $R$ and $S$ admitting semisimple Nakayama Automorphisms, and $f \in C^*(R, R)$, $g \in C^*(S,S)$, the $\bv$-operator on $\HH^*(R\otimes S, R\otimes S)$ is given by
$${\Delta}(f\boxtimes g) =  (-1)^{mn}(\Delta(f) \boxtimes g +(-1)^{n} f\boxtimes \prescript{}{}{\Delta(g)})$$
\end{corollary*}
which generalises the results of \cite{Zhou} by upgrading from requiring that the Nakayama Automorphism is the identity.
  \section{Twisted Tensor Products of Frobenius Algebras}
\subsection{Basic Notions} Take $R$, $S$ finite dimensional $k$-algebras graded by Abelian Groups $A$ and $B$ respectively, so that we may write
$$R = \bigoplus_{a\in A} R^a \hspace{20pt} S = \bigoplus_{b \in B} S^b$$
Given $a \in A$, $b \in B$, and $r \in R^a$, $s \in S^b$, we write $a=:|r|, b=:|s|$. We also have a \emph{twisting}, (often called a bicharacter), which is a group homomorphism $t: A \otimes B \rightarrow k^{\times}$. We will call the values taken on by $t$ twisting coefficients. We define the algebra $R \otimes^t S$, whose underlying vector space is $R \otimes S$, and whose multiplication is given by 
$$(r \otimes s) \times^t (r' \otimes s') = t(|r'|,|s|) rr'\otimes ss'$$
for $r,r' \in R$, $s,s' \in S$ homogeneous. We abuse notation and write $t(r',s) = t(|r'|, |s|)$, as is standard. From here on, suppose $R$ and $S$ are both Frobenius Algebras; that is, they admit inner products
$$\langle -, - \rangle_R : R \otimes R \rightarrow k \hspace{20pt}\langle -, - \rangle_S : S \otimes S \rightarrow k$$
which are non-degenerate (so that $ x\mapsto \langle -,  x \rangle_{R}$ is injective), and have the following invariance property:
$$\langle ab, c \rangle = \langle a, bc \rangle$$ 
\begin{examples} 
Take $\Lambda(n) = k[x]/x^n$, which admits a grading by $\mathbb{Z}$ where $|x|=1$. $\Lambda(n)$ admits a Frobenius structure via
$$\langle x^i, x^j \rangle = \begin{cases} 1 \textrm{ if } i+j=n \\ 0 \textrm{ otherwise } \end{cases}$$
Choosing any $q \in k^\times$, we can choose a twisting map $t:\mathbb{Z} \otimes \mathbb{Z} \rightarrow k^\times$,  via the following formula:
$$t(i\otimes j) = q^{ij}$$
Constructing $\Lambda_q(n,m) = \Lambda(n) \twisten \Lambda(m)$, to obtain the first example of the \emph{quantum complete intersections} that motivate our study. Iterating this construction, (with adjusted twisting maps for more variables) gives rise to all such quantum complete intersections.
\end{examples}
\subsection{Frobenius Structure on Twisted Tensor Products}
On homogeneous elements $r,r' \in R$, $s,s'\in S$, define:
$$\langle r \otimes s ,r' \otimes s' \rangle_{R\otimes^t S}= t(r',s)\langle r,r' \rangle_R \langle s, s' \rangle_S$$
To lighten the load of notation, we will write $\langle-,- \rangle := \langle -,- \rangle_{R \otimes^t S}$. 
\begin{proposition} $\langle-,- \rangle$ makes $R \otimes^t S$ into a Frobenius Algebra. \end{proposition}
\begin{proof} We first must show that $ \langle-,- \rangle$ is a well defined morphism $R \otimes^t S \otimes R \otimes^t S \rightarrow k$. We note that restricting our morphism to the direct summands of the graded decomposition
$$R^{a} \otimes S^{b} \otimes R^{a'} \otimes S^{b'} \rightarrow k$$
our map is simply the composition
$$R^{a} \otimes S^{b} \otimes R^{a'} \otimes S^{b'} \xrightarrow{\cdot t(a',b)} R^{a} \otimes R^{a'} \otimes S^{b} \otimes S^{b'} \xrightarrow{\langle-,-\rangle_R \otimes\langle -,- \rangle_S} k\otimes k \xrightarrow{\mu} k$$
which is clearly linear. 
We next check the Frobenius Condition. This is the requirement that
$$ \langle (a \otimes b) \times (c\otimes d), e \otimes f \rangle = \langle a \otimes b, (c \otimes d) \times (e \otimes f)\rangle $$
and indeed, enduring some algebra
\begin{align*}&\langle ( a \otimes b ) \times^t (c \otimes d ), e \otimes f\rangle\\
= \;\;&\langle t(c,b)(ac \otimes bd), e\otimes f\rangle \\
=\;\;&t(c,b)t(e,bd)\langle ac, e \rangle_R \langle bd, f \rangle_S\\
=\;\;&t(c,b)t(e,b)t(e,d)\langle a, ce\rangle_R \langle b, df \rangle_S \\
= \;\;&t(ce,b)t(e,d)\langle a, ce \rangle_R \langle b, df \rangle_R \\
=\;\; & t(e,d) \langle a \otimes b, ce \otimes df\rangle \\
=\;\; & \langle a \otimes b, (c \otimes d) \times^t (e \otimes f)\rangle
\end{align*}
we see this is the case. \\[5pt]
We now only need to show our inner product is non-degenerate.
We know that both $R$ and $S$ are finite dimensional (as the Frobenius condition forces $R \cong D(R)$), so we may choose bases $\{ e_1,\cdots, e_n\}$ and $\{f_1, \cdots, f_m\}$ respectively, and suppose that these bases are homogeneous. Further, as each of $\langle -,- \rangle_R$ and $\langle -,- \rangle_S$ are non-degenerate, we can choose (left/right) dualizing\footnote{This choice of phrase is to avoid confusion between dual bases, which live in the dual.} bases, $\{\epsilon_1, \cdots, \epsilon_n\}$ and $\{ \phi_1, \cdots, \phi_m\}$ so that
$$\langle \epsilon_i, e_j \rangle_R = \delta_{i,j}1_k \textrm{ (left) } \hspace{20pt} \langle f_i, \phi_j \rangle_S = \delta_{i,j}1_k \textrm{ (right) }$$
To construct these explicitly, we take what is often called the ``dual'' basis of $\{e_1,...e_n\}$ in $D(A)$; i.e, the morphisms
$$d_i(e_j) = \delta_{i,j} 1_k$$
and consider the inverse image of this basis under the isomorphisms $(x \mapsto \langle x, - \rangle_R)$ and $(y \mapsto \langle - , y \rangle_S)$. Note, we have chosen dualizing bases in different coordinates for $R$ and $S$, to make our proof easier. \\[5pt]
Suppose we have $x \in R \twisten S$ such that
$$\langle  x, a \otimes b \rangle = 0 \textrm{ for all } a\otimes b \in R\otimes^t S.$$
Then we know that $\{ \epsilon_i \otimes f_j \}_{i,j}$ form a basis of $R \twisten S$, so write $x = \sum_{i,j} \lambda_{i,j} \epsilon_i \otimes f_j$.
\begin{align*}0= \langle x, e_{i_0} \otimes \phi_{j_0} \rangle&= \sum_{i,j}\lambda_{i,j}\langle \epsilon_i \otimes f_j, e_{i_0} \otimes \phi_{j_0} \rangle \\
&= \sum \lambda_{i,j} t(f_j, e_{i_0})\langle\epsilon_i, e_{i_0}\rangle_R \langle f_j, \phi_{j_0} \rangle_S \\
&= \lambda_{i_0,j_0}t(f_{j_0}, e_i).
\end{align*}
So then, as $t: A \otimes B \rightarrow k^{\times}$, we know that $t(f_{j_0}, e_i)\not=0$, and thus $\lambda_{i_0,j_0}=0$. From here it is clear that $x=0$.
\end{proof}
\section{Semisimplicity of the Nakayama Automorphism}
\subsection{Grading and the Nakayama Automorphism}
\begin{notation} In future, for convenience when denoting the inner product in diagrams, we will write
$$\langle \cdot, - \rangle:= ( a \mapsto \langle a, - \rangle )$$
$$\langle -, \cdot \rangle:= ( a \mapsto \langle -, a \rangle ) $$
\end{notation}
We have shown that $R \otimes^t S$ is a Frobenius Algebra. We now seek to show that should $R$ and $S$ both admit diagonalisable Nakayama Automorphisms, then $R \otimes^t S$ does also. To do this, we will need to introduce an assumption on the interaction between the grading and the inner product. We will ask that the morphisms
$$\langle - ,- \rangle_R : R \otimes R \rightarrow k \hspace{20pt} \langle -,- \rangle_S: S \otimes S \rightarrow k$$
are such that there exist elements $\sigma_R \in A$, and $\sigma_S \in B$ so that
$$R \otimes R \rightarrow k[\sigma_R], \hspace{20pt} S \otimes S \rightarrow k[\sigma_S]$$
are graded, where we view $k$ as concentrated in degree $0$, and we have that $k[\sigma_R]^{a} =k^{a+\sigma_R}$. We know for $a, a' \in A$, we have
$$R^a \otimes R^{a'} \rightarrow (k[\sigma_R])^{a+a'}=k^{a+a'+\sigma_R}.$$
Then, should $r, r'\in R$ be homogeneous elements with non-vanishing inner product, we must have $|r|+|r'|+\sigma_R=0$. We ask for this assumption as it is satisfied by our most important examples, and it forces the Nakayama Automorphism to be graded.
\begin{examples} 
Given $\Lambda(n)=k[x]/x^n$ as defined before, we recall that
$$\langle x^i, x^j \rangle = \begin{cases} 1 \textrm{ if } i+j=n \\ 0 \textrm{ otherwise. } \end{cases}$$
But then we have, for $r,r' \in \Lambda(n)$ homogeneous 
$$\langle r, r' \rangle \not= 0 \implies |r|+|r'|+(-n)=0$$
So that in this case, we have such a shift, equal to $-n$. 
\end{examples}
As $R$ is a graded algebra, $D(R)$ is also naturally graded, as follows:
$$D(R)_m := \{ f : R \rightarrow k: f(r) =0 \textrm{ for } r \not \in R^{-m}\}$$
We note that the statements that the following maps are graded
$$\begin{aligned} \langle \cdot, - \rangle_R, \langle -, \cdot \rangle_R : R \rightarrow D(R)[\sigma_R] \end{aligned}$$
is exactly the requirement that
$$\langle r, r' \rangle \not= 0 \implies |r'|+\sigma_R =-|r|$$
Which is equivalent to our earlier condition. Note that this condition does not always hold. The inner product is uniquely determined up to inner automorphism on $R$ (and $S$); but then, acting by  a non-homogeneous unit will interfere with the grading. \\

We briefly recall the properties of the Nakayama Automorphism. For a Frobenius algebra $R$, coming equipped with a distinguished inner product $\langle -, - \rangle_R$, the Nakayama Automorphism $\nu_R$ is defined so that:
$$\langle a ,- \rangle_R = \langle -,\nu_R(a) \rangle_R$$
which encodes that $\nu$ is the composition of the following two maps
$$R \xrightarrow{ \langle \cdot , - \rangle_R} D(R) \xrightarrow{ (\langle -, \cdot \rangle_R)^{-1}} R$$
It is clear from the composition above that $\nu_R$ is an automorphism of vector spaces, and also easy to check that $\nu_R$ is a ring automorphism. Moreover, we have the following crucial result.
\begin{lemma} The morphism $R \rightarrow D(R)$ given by $r \mapsto \langle -, r \rangle_R$ is a bimodule isomorphism where the right action of $R$ on itself is via $\nu_R$. That is to say,
$$R_{\nu_R} \cong D(R)$$
\end{lemma}
We will prove a more general version of this statement later. 
\begin{proposition} Suppose $R$, $S$ and $\langle - , - \rangle_R$, $\langle -,- \rangle_S$ obey the grading properties above. Then 
\begin{enumerate} 
\item The Nakayama Automorphisms $R \rightarrow R$, $ S \rightarrow S$ preserve gradings.
\item The morphism $R \otimes^t S \rightarrow D(R\otimes^t S)[(\sigma_S, \sigma_R)]$ is graded.
\end{enumerate} 

\end{proposition}
\begin{proof}[Proof of 1] 
We first verify that both morphisms
$$R \mapsto D(R)[\sigma_R]$$
$$ r \mapsto \langle r, - \rangle, \; \; r \mapsto \langle - , r \rangle$$
are graded linear maps. But this is just the statement that
$$\langle x, r \rangle_R \not=0 \implies |x|+|r|+\sigma_R=0, \;\;  \langle r,x \rangle_R \not=0 \implies |x|+|r|+\sigma_R=0$$
which are exactly the assumptions on the map $R \otimes R \rightarrow k$. So then obviously we have morphisms
$$R \xrightarrow{\langle \cdot, - \rangle_R} D(R)[\sigma_R] \xrightarrow{\langle -, \cdot \rangle_R^{-1}} R,$$
as inverting the second morphism reverses the grading shift. The Nakayama automorphism is this composition, which is clearly graded.
\end{proof}
\begin{proof}[Proof of 2]
We know for $r, a \in R$, $s, b \in S$ homogenous that 
$$\langle r \otimes s, a \otimes b \rangle = t(a,s)\langle r, a \rangle_R \langle s, b \rangle_S.$$
Now, $\langle r, a \rangle_R$ vanishes when $|a| \not = -|r|-\sigma_R$, and likewise $\langle s, b \rangle_S=0$ when $|b|\not=-|s|-\sigma_S$. So the inner product vanishes when $| a\otimes b| \not = (-|r| - \sigma_R, -|s| - \sigma_S)$, exactly as required.
\end{proof}
\subsection{Towards Diagonalisability}
We aim to show that if $R$ and $S$ admit semisimple Nakayama Automorphisms, then $R \otimes S$ also admits a semisimple Nakayama Automorphism. We recall that a morphism $\phi: V \rightarrow V$ of $k$-vector spaces is semisimple when
$$\bar{\phi}: V \otimes_k \bar{k} \rightarrow V \otimes_k \bar{k}$$ is diagonalisable, where $\bar{k}$ is an algebraic closure of $k$. We first prove the following Lemma.
\begin{lemma} Given $R$, $S$ $k$-algebras graded by abelian groups $A$ and $B$ respectively, with a twisting morphism $t: A \otimes B \rightarrow k^\times$,
\begin{enumerate}
\item $(R \otimes_k \bar{k}) \otimes_{\bar{k}} (S \otimes_k \bar{k}) \cong (R \otimes_k^t S) \otimes_k \bar{k}$  as vector spaces
\item Letting $\bar{t}$ denote the composition $\bar{t}: A\otimes B \rightarrow k^\times \rightarrow \bar{k}^\times$ we have
$$( R \otimes_k \bar{k}) \otimes^{\bar{t}}_k (S \otimes_k \bar{k}) \cong (R \otimes^t_k S)\otimes_k \bar{k}$$
as algebras, where $R\otimes \bar{k}$ and $S \otimes \bar{k}$ are graded in the obvious ways.
\end{enumerate}
\end{lemma}
\begin{proof}[Proof of 1.]

We have a morphism from the right to the left, for $r \in R$, $s\in S$ and $\lambda, \mu \in \bar{k}$
$$ (r \otimes \lambda) \otimes (s \otimes \mu) \mapsto r \otimes s \otimes \lambda\mu$$
with inverse
$$(r \otimes s )\otimes \lambda  \mapsto (r \otimes \lambda) \otimes (s \otimes 1).$$
We must check that the first morphism is well defined. First we note that choosing $\alpha \in \bar{k}$
\begin{center}
\begin{tikzcd}
( r \otimes \lambda \alpha ) \otimes ( s \otimes \mu) \arrow[r, mapsto] \arrow[d, equal] & r \otimes s \otimes \lambda\mu\alpha \arrow[d, equal]  \\
 (r \otimes \lambda) \otimes (s \otimes \mu \alpha) \arrow[r, mapsto] & r \otimes s \otimes \mu \lambda \alpha.
 \end{tikzcd}
 \end{center}
So that this morphism is in fact well defined. These morphisms are obviously mutually inverse, so that we are done. 
\end{proof}
\begin{proof}[Proof of 2.] We must check the vector space isomorphism is an algebra morphism. Take $r,r' \in R$, $s,s' \in S$ homogeneous, and $\lambda, \mu, \lambda', \mu' \in \bar{k}$
\begin{align*}
(r \otimes \lambda) \otimes (s \otimes \mu) \times (r' \otimes \lambda') \otimes (s' \otimes \mu') &= t'(r',s) (rr' \otimes \lambda\lambda') \otimes (ss' \otimes \mu\mu') \\
&\mapsto (t(r',s) rr' \otimes ss') \otimes \lambda \lambda' \mu \mu' \\
&= ((r \otimes s) \otimes \lambda\mu) \times ((r' \otimes s') \otimes \lambda'\mu').
\end{align*}

So our isomorphism respects multiplication. 
\end{proof}
\begin{remark} Note that all gradings commute with the tensor product; for $R$ an $A$-graded $k$-algebra, $R\otimes_k \bar{k}$ becomes an $A$-graded $\bar{k}$-algebra. It is also easy to see that $R \otimes_k \bar{k}$ remains Frobenius, by writing down a non-degenerate invariant form. 
\end{remark}
So to show that $R \otimes^t S$ admits a semisimple Nakayama Automorphism, it suffices to show that for $R$, $S$, admitting diagonalisable Nakayama Automorphisms, $R\otimes^t S$ admits a diagonalisable Nakayama Automorphism. \\[5pt]
Before stating our final proposition of the section, we briefly recap the properties of adjoints. Given a vector space $V$ admitting a non degenerate bilinear form $\langle -, - \rangle$, and an endormorphism $T: V\rightarrow V$, the adjoint of $T$ written $T^*$ is the linear map with the property that 
$$\langle Ta, b \rangle = \langle a, T^*b \rangle$$
which can be described explicitly by the composition
$$V \xrightarrow{\langle -, \cdot \rangle} D(V) \xrightarrow{\cdot T} D(V) \xrightarrow{(\langle -, \cdot \rangle)^{-1}} V$$
For $V$ finite dimensional, should we have $\{e_1,\cdots, e_n\}$ a basis of eigenvectors for $T$, with eigenvectors $\{\lambda_1, \cdots, \lambda_n\}$, whose (right) dualizing basis is $\{\epsilon_1, \cdots, \epsilon_n\}$, then $T^*$ will act on $\epsilon_i$ by $\lambda_i$. To see this, note that we have
$$\begin{aligned} \langle e_j, T^*\epsilon_i \rangle = &\langle Te_j, \epsilon_i \rangle\\
= &\lambda_j \langle e_j, \epsilon_i \rangle \\
= & \lambda_j \delta_{i,j} \\
= &\lambda_i \delta_{i,j} = \langle e_j, \lambda_i\epsilon_i \rangle,\end{aligned} $$
so that, by non-degeneracy of our inner product, $T^*\epsilon_i = \lambda_i \epsilon_i$. 
\begin{proposition} Suppose $R$ and $S$ are Frobenius Algebra, with the grading assumptions above. We have:
$$\nu(a \otimes b) = t(|a|, \sigma_S)t(\sigma_R, |b|)^{-1} \nu_R(a) \otimes \nu_S(b),$$
where we have written $\nu$ for $\nu_{R\twisten S}$.
\end{proposition}
\begin{proof}
Take $a, c \in R$, $b,d \in S$ homogeneous. Writing $\nu_R, \nu_S$ for the Nakayama Automorphisms of $R,S$ respectively we have
\begin{align*}
\langle a \otimes b, c \otimes d\rangle &= t(c,b) \langle a ,c \rangle_R \langle b,d \rangle_S \\
&=t(c,b) \langle c, \nu_R(a) \rangle_R \langle d, \nu_S(b) \rangle \\
&= \frac{t(c,b)}{t(a, d)} \langle c\otimes d, \nu_R(a) \otimes \nu_S(b) \rangle,
\end{align*}
where on the last line we used the fact that $\nu_R$ is homogeneous. So we see that
$$\langle a \otimes b , -  \rangle = \left\langle\frac{t(-,b)}{t(a, -)} -, \nu_R(a) \otimes \nu_S(b)\right\rangle$$
where the morphism $T:=\frac{t(-,b)}{t(a, -)}-$ maps the homogeneous element 
$$ c \otimes d \mapsto \frac{t(c,b)}{t(a,d)} c \otimes d.$$
So, we see clearly that\footnote{Here crucially $T$ depends on $a$ and $b$.}
$$\nu (a\otimes b) = T^* (\nu_R(a) \otimes \nu_S(b))$$
Now, we know that $T$ is diagonalised by any basis of homogeneous elements. We also know that, should a homogeneous element $r \otimes s$ be part of some homogeneous basis $\mathcal{B}$, then the corresponding element of the (left) dualizing basis $x$ will also be homogeneous, and have
$$ \langle x, r\otimes s \rangle \not =0$$
So that, $|x|_A = -|r| -\sigma_R$ and $|x|_B =-|s| - \sigma_S$. From our discussion earlier, we know that $T^*$ will act on $r\otimes s$ via the scalar $\frac{t(|x|_A,b)}{t(a,|x|_B)}$. That is, for $r, s$ homogeneous,
$$T^*( r \otimes s) = \frac{t(-|r|-\sigma_R, |b|)}{t(|a|, -|s|-\sigma_S)} r \otimes s$$
which finally tells us that

\begin{align*}
\nu(a \otimes b) &= \frac{t(-|a|-\sigma_R, |b|)}{t(|a|, -|b|-\sigma_S)} \nu_R(a) \otimes \nu_S(b) \\
&= t(|a|, \sigma_S)t(\sigma_R, |b|)^{-1} \nu_R(a) \otimes \nu_S(b)
\end{align*}
as required.
\end{proof}
\begin{corollary} Suppose $R$ and $S$ satisfy the grading criteria, and further have diagonalisable Nakayama Automorphisms. Then the Nakayama Automorphism of $R\otimes^t S$ is diagonalisable.
\end{corollary}
\begin{proof}
The grading assumptions show that the Nakayama Automorphisms restrict to each graded piece. Recall that a linear map is diagonalisable exactly when its minimal polynomial has no repeated factors. We see that the minimal polynomial of a linear map restricted to a subspace will always divide the original minimal polynomial, so the restriction of a diagonalisable morphism to an invariant subspace remains diagonalisable. So then, each restriction of the Nakayama Automorphism admits a diagonalising homogeneous basis, which will clearly be a diagonalising basis for the twisted Nakayama Automorphism.
\end{proof}
\section{New information from old}
\subsection{Modules}
The objective of our study is primarily to transport well known information about the categories $R\bimods$ and $S\bimods$, into $R \twisten S \bimods$. A first step is the creation of a functor between $R\bimods \times S\bimods \rightarrow R \twisten S\bimods$. Crucially, for compatibility with our twisting, we must have some notion of grading of our modules. An $R$-module $M$ is graded when it admits a decomposition
$$M = \bigoplus_{a \in A} M^a$$
for which with $r\in R^a$ and $m \in M^b$ $rm \in M^{a+b}$. From here on, all modules are assumed to be graded, as are all morphisms.

\begin{defn}Given $M \in R\bigmods$, $N \in S\bigmods$, we denote $M \twisten N \in R \twisten S \bigmods$ for the module with underlying vector space $M \ten N$, with action given as follows. For $r \in R$, $s \in S$, $m \in M$ and $n \in N$ all homogeneous, 
$$(r \otimes s) \cdot(m \otimes n) = t(|m|,|s|)(rm \otimes sn)$$
$$ (m \otimes n)\cdot(r \otimes s) = t(|r|,|n|)(mr \otimes ns)$$
It is for example easy to verify that the bimodule structure of $R \twisten S$ agrees with this prescription.
\end{defn}
For a functor, we further require a translation of morphisms. 
\begin{proposition}
Given morphisms $M \xrightarrow{f} M' \in R\bigmods$ and $N \xrightarrow{g} N' \in S\bigmods$ (which we recall must be graded), we have the morphism
$$M \twisten N \xrightarrow{f \twisten g} M' \twisten N',$$
which is the same on objects as the usual morphism $f \twisten g$. This is a morphism of $R \twisten S$-bimodules.

\end{proposition}
\subsection{Resolutions}
Given projective bimodule resolutions $P_\bullet \rightarrow R$, and $Q_\bullet \rightarrow S$, it is known (see \cite{Witherspoon}) that $\operatorname{Tot}(P_\bullet \twisten Q_\bullet)$ is a projective bimodule resolution of $R \twisten S$. The resolutions of interest to us are the reduced and unreduced Bar Resolutions, which might be described as follows. Write $\bar{R} = R / (k \cdot 1_R)$ 
$$B_n(R) := \bar{R}^{\otimes n} \hspace{20pt} \tilde{B}_n(R) := R^{\otimes n}$$
and write $[r_1|\cdots |r_n] := r_1 \ten \cdots \ten r_n \in B_n(R)$, following the notation of \cite{BriggsWitherspoon}. We have\footnote{ We will focus on $B_n(R)$, the reduced resolution as we use it more frequently; we will only need the unreduced resolution to construct comparison maps.}
$$b_R:B_n (R) \rightarrow B_{n-1}(R), \; \; b_R([r_1|\cdots|r_n]) = \sum_{i=1}^{n-1} (-1)^i[r_1| \cdots |r_i r_{i+1}| \cdots | r_n]$$

Write $\mathbb{B}(R)$ for the resolution given by $\mathbb{B}(R)_n = R \otimes B_n R \otimes R$, with differential defined $R$-bilinearly by
$$\delta_R(1 \otimes [r_1 |\cdots | r_n] \otimes 1) = r_1 \otimes [r_2 | \cdots |r_n] \otimes 1 +  1 \ten b_R([r_1|\cdots | r_n])\ten 1 +(-1)^n 1 \ten [r_1 | \cdots |r_{n-1}] \ten r_n$$
and write $\mathbf{B}(R)$ for the resolution $\mathbf{B}(R)_n = R \otimes \tilde{B}_n(R) \otimes R$, with similar differential. We have similarly, $B(S)$, $\tilde{B}(S)$, $\mathbb{B}(S)$ and $\mathbf{B}(S)$ wherein identical differentials are used. By the previous discussion, $\mathbb{B}(R) \otimes^t \mathbb{B}(S)$ is an $R \twisten S$-bimodule projective resolution of $R \twisten S$, and moreover, from \cite{BriggsWitherspoon} we have bimodule isomorphisms
$$\mathbb{B}(R) \otimes \mathbb{B}(S) \cong (R \twisten S) \otimes BR \otimes BS \otimes (R \twisten S)$$
$$(1 \otimes [r_1 | \cdots | r_n] \ten 1) \ten (1 \ten [s_1 |\cdots |s_m] \ten 1) \mapsto (1 \ten 1) \otimes [r_1 | \cdots | r_n] \ten [s_1 |\cdots |s_m] \ten (1 \otimes 1)$$
These resolutions were used to great effect in $\cite{BriggsWitherspoon}$ to make explicit both the product and the Gerstenhaber Bracket, in terms of the operations solely on the Hochschild Cohomologies of $R$ and $S$.
\subsection{Hochschild Cohomology}
We are now nearly able to state the main result motivating our work. Firstly we will solidify some notation that will be crucial later. Recall from the previous section, that we were given $k$-algebras $R$ and $S$ admitting respectively a grading by abelian groups $A$, $B$. We are further given a twisting map $t: A \otimes B \rightarrow k^{\times}$. Following the notation of \cite{BriggsWitherspoon} and writing $\hat{A} = \homo_{\textbf{Ab}} (A, k^\times)$ for the group of linear characters of $A$, we note that $\hat{A}$ acts on $R$. Given $\lambda \in \hat{A}$ and $r \in R$ homogeneous, we can define $\lambda \cdot r = \lambda(|r|) r$. Most crucially we note that
$$t \in \homo_{\textbf{Ab}}(A \otimes B, k^{\times}) \cong \homo_{\textbf{Ab}}(B, \hat{A})$$
So to each $b \in B$ we can associate $\hat{b} = t(-,b)$, which acts on $R$ by the above formula. The same of course holds with $B$ and $A$ swapped. \\

Continuing on from the last section, we have, 
$$\begin{aligned} C_t^*(R \twisten S, R \twisten S) &:=  \homo_{R \twisten S ^e} ( (R \twisten S) \otimes BR \otimes BS \otimes (R \twisten S), R \twisten S)\\& \cong \homo_k(BR \otimes BS, R \otimes S)\end{aligned}$$
and we further have the following morphism:
$$\homo_k(BR, R) \otimes \homo_k(BS, S) \xrightarrow{- \boxtimes -} \homo_k(BR \otimes BS, R \otimes S)$$
$$(f \boxtimes g)([r_1 | \cdots | r_n] \otimes \els) =(-1)^{mn}f([r_1 | \cdots | r_n]) \otimes g(\els)$$
for $f \in \homo_k (BR, R)$, $g \in \homo_k(BS, S)$. This map is clearly an injection and as everything in sight is finite dimensional comparing dimensions ensures that our map is an isomorphism, which further preserves the $A \oplus B$ grading. Comparing graded pieces and verifying that our map is in fact a chain morphism \cite{BriggsWitherspoon} we obtain the following theorem:
\begin{theorem}\label{TwistedTensorTheorem} With $R$, $S$, $A$ and $B$ as above

$$\HH^*(R \twisten S, R \twisten S) \cong \bigoplus_{a \in A, b \in B} \HH^*(R, R_{\hat{b}})^a \otimes \HH^*(S, \prescript{}{\hat{a}}{S})^b$$

of graded vector spaces.
\end{theorem}
Then Briggs and Witherspoon proceed to show that this is in fact an isomorphism of Gerstenhaber Algebras, giving a Gerstenhaber Structure on the right hand side by defining a notion of \emph{twisted circle product} on the complexes $C^*(R,R_{\hat{b}})$ and $C^*(S, \prescript{}{\hat{a}}S)$. In the coming sections we will endeavour to demonstrate corresponding \emph{twisted $\bv$ operators} on the same complexes.

\section{The BV operator}
In this section we will study constructions of the $\bv$-operator, and describe the twisted constructions that show up in our main theorem. Before we begin discussing these, we must define exactly what we mean by a $\bv$-operator. 
\begin{defn} A Gerstenhaber algebra over a field $k$ is a triple $(\mathcal{H}, \smile, [-,-])$ with the following properties. $\mathcal{H} = \oplus_{n\in\mathbbm{N}}\mathcal{H}^n$ is a graded $k$-vector space, equipped with graded maps
$$\smile: \mathcal{H} \otimes \mathcal{H} \rightarrow \mathcal{H}$$
$$[-,-]: \mathcal{H} \otimes \mathcal{H} \rightarrow \mathcal{H}[-1]$$
the latter of which is called the Gerstenhaber bracket. These obey
\begin{enumerate}[label=(\roman*)]
\item $(\mathcal{H}, \smile)$ is a graded commutative algebra with unit.
\item For $f,g,h \in \mathcal{H}$ homogeneous elements
$$ [f,g] = -(-1)^{(|f|-1)(|g|-1)}[g,h]$$
$$0= (-1)^{(|f|-1)(|h|-1)}[[f,g],h] + (-1)^{(|g|-1)(|f|-1)}[[g,h],f] + (-1)^{(|h|-1)(|g|-1)}[[h,f],g]$$
\item For each $f \in \mathcal{H}$ the map $[f,-]$ is a graded derivation on the algebra $(\mathcal{H}, \smile)$.
\end{enumerate}
\end{defn}
\begin{defn} A Batalin-Vilkovisky ($\bv$) algebra is a quadruple $(\mathcal{H}, \smile, [-,-], \Delta)$ with the following properties. $(\mathcal{H}, \smile, [-,-])$ is Gerstenhaber algebra, and $\Delta:\mathcal{H} \rightarrow \mathcal{H}[-1]$ is an operator which squares to zero. For $f,g \in \mathcal{H}$ homogeneous, $\Delta$ also satisfies the following equation
$$[f,g] = (-1)^{|f|+1}(\Delta(f \smile g) - \Delta(f) \smile g - (-1)^{|f|} f \smile \Delta(g)).$$
So then, the Gerstenhaber bracket is completely determined by the product structure and the $\bv$-operator.
\end{defn}
\subsection{Tamarkin-Tsygan Calculi} In this section we will briefly recall the construction of a Tamarkin-Tsygan Calculus that is used to define the $\bv$-operator on Hochschild Cohomology in \cite{Lambre}. Though the precise details are not crucial, our twisted $\bv$-operators will arise from this construction. \\[5pt]
Let $A$ be an associative finite dimensional algebra, and $\nu:A \rightarrow A$ be an algebra automorphism. Recall that $C_p(A, M):= M \otimes_{A^e} \mathbb{B}(A)_p$. We will denote chains as follows:
$$(m, a_1, \cdots, a_p) := m \otimes_{A^e} (1 \otimes a_1 \otimes \cdots \otimes a_p \otimes 1)$$
with this notation fixed, define  $\beta_\nu:C_p(A, A_\nu) \rightarrow C_{p+1}(A,A_\nu)$
$$(a_0, a_1, \cdots, a_n) \mapsto \sum_{i=0}^p (-1)^{ip} (1 , a_i , \cdots , a_p , a_0 , \nu(a_1) , \cdots , \nu(a_{i-1}))$$
and let $T:C_p(A, A_\nu) \rightarrow C_p(A, A_\nu)$ denote the morphism
$$(a_0, a_1, \cdots, a_p) \mapsto (\nu(a_0), \nu(a_1), \cdots, \nu(a_p)).$$
Then from \cite{Lambre} we have the following proposition.
\begin{proposition} Writing $\delta$ for the differential on $C_*(A, A_\nu)$, we have
\begin{equation}\label{betaDiff} \delta \beta_\nu + \beta_\nu \delta = 1-T \end{equation}
\end{proposition}
Suppose for now that $\nu$ is diagonalisable, with eigenvalues $\Lambda$. Then, levelwise, $T$ is certainly diagonalisable, and our chains split into complexes $C^\mu_\bullet(A,A_\nu)$, the $\mu$-eigenspaces of $T$. From here, we notice that should we restrict to the $1$-eigenspace, the right hand side of Equation \ref{betaDiff} vanishes, so that $\beta_\nu$ is a chain map restricted to this subcomplex. Further, considering any $\mu$-eigenspace for $\mu\not=1$, $\beta_\nu$ demonstrates that the identity is nill-homotopic, so that 
$$\operatorname{H}_*(C^1_*(A, A_\nu)) = \HH_*(A, A_\nu).$$
We denote the operator $B_\nu :=\beta_\nu|_{C^1_\bullet(A,A_\nu)}$, which is our analogue of the Connes Operator for constructing the BV operator. \\
\subsection{Untwisted $\bv$-operators}
In \cite{Lambre}, they proceed to use the duality of Hochschild Homology and Cohomology for Frobenius Algebras to transport the morphism $B_\nu$ into a morphism on Hochschild Cohomology, which defines the $\bv$-operator. Then more work is done to ensure that they can transfer from the diagonalisable case to the semisimple case. We will first outline this approach for clarity, before defining our twisted $\bv$-operators. In this subsection, suppose that $A$ is Frobenius and that $\nu$ is the Nakayama Automorphism of $A$. We have the following sequence of chain isomorphisms:
$$\begin{aligned}D(C_*(A,A_\nu))&=\homo_k(A_\nu \otimes_{A^e} \mathbbm{B}_\bullet (A), k) \\
& \cong \homo_{A^e}(\mathbbm{B}_\bullet (A), D(A_\nu)) \\
& \cong \homo_{A^e}(\mathbbm{B}_\bullet (A), A) = C^*(A,A),
\end{aligned}$$
where we have used the map $a \mapsto \langle a, - \rangle$ for the map $A \xrightarrow{\cong} D(A_\nu)$. Write $\partial:D(\HH_*(A,A_\nu)) \xrightarrow{\cong} \HH^*(A,A)$ for the isomorphism induced on (Co)homology, which we will abuse to also write on chains and cochains. We can work through this chain of isomorphisms to describe explicitly the pairing this induces on Homology and Cohomology; given $\alpha \in C^p(A,A)$ and $(a_0, a_1, \cdots, a_n) \in C_p(A,A_\nu)$, we can write 
$$\begin{aligned}
\partial^{-1}(\alpha)((a_0, a_1, \cdots, a_p))= \langle  \alpha([ a_1 | \cdots | a_p ]) , a_0\rangle
\end{aligned}$$
\begin{remark} Note that this formulation differs from the one written down in \cite{Lambre} by the orientation of the inner product. In their paper they write down a morphism $A_\nu \rightarrow D(A)$, but in the sequence of isomorphisms above use $A \rightarrow D(A_\nu)$. This dualising ammounts to switching the inner product; though this is not important in their paper, it becomes crucial for the careful consideration of twisting coefficients in our theorem.\end{remark}
\begin{proposition} Supposing that $\nu$ is diagonalisable, the operator $\Delta$ defined via the following sequence of maps:
$$\begin{aligned}\Delta: \HH^*(A,A) &\xrightarrow{\partial^{-1}} D(\HH_*(A,A_\nu)) \cong D(\HH_*^1(A, A_\nu)) \xrightarrow{D(B_\nu)} \cdots \\
\cdots &\rightarrow D(\HH_{*-1}^1(A, A_\nu) \cong D(\HH_*(A,A_\nu)) \xrightarrow{\partial} \HH^{*-1}(A,A)
\end{aligned}
$$
gives $\HH^*(A,A)$ the structure of a $BV$ algebra. 
\begin{proof} See \cite{Lambre}. \end{proof}
\end{proposition}
And to complete the story, we have the following:
\begin{theorem} Given $A$ a $k$-algebra admitting a semisimple Nakayama Automorphism, with $\bar{k}$ an algebraic closure of $k$, 
$$\HH^*_{\bar{k}}(A \otimes \bar{k}) \cong \HH^*(A) \otimes \bar{k}$$
and the restriction of the $\bv$-operator to the subspace $\HH^*(A) \otimes 1$ gives $\HH^*(A)$ the structure of a $\bv$-algebra.
\end{theorem}
\begin{proof} See \cite{Lambre}. \end{proof}
For clarity, and for later usage we will briefly make explicit a ``formula'' for the $\bv$-operator. The main difficulty we face is that the map $\partial$ is rather hard to write down. To avoid this we have the following
\begin{lemma} Given $f \in \HH^n(A,A)$ we have 
$$\begin{aligned} \langle \Delta(f)([a_1 |\cdots | a_n ]), a_0\rangle =  \langle  \sum_{i=0}^n (-1)^{in}f([ a_i | \cdots | a_n | a_0 | \nu(a_1) | \cdots | \nu(a_{i-1}) ]), 1 \rangle \end{aligned}$$
\end{lemma}
\begin{proof}
 Take $f\in \HH^n(A,A)$, which we will view as a function $f \in \homo_{k}({B}_n(A), A)$. Then, we know that $\Delta(f) = \partial \circ D(B_v) \circ \partial^{-1} f$. To avoid dealing with $\partial$, we can write this
$$\partial^{-1}(\Delta(f)) = D(B_v) \circ \partial^{-1}(f)$$
So that, taking $(a_0, a_1, \cdots, a_n) \in C_n(A,A_\nu)$  we have on the left side of our equation
$$\begin{aligned}\partial^{-1}(\Delta(f))(a_0, a_1, \cdots, a_n) &=  \langle  \Delta(f)( [ a_1 | \cdots | a_n ]), a_0 \rangle,
\end{aligned}$$
and on the right
$$\begin{aligned} D(B_v) \circ \partial^{-1}(f)(a_0,a_1,\cdots, a_n) &= \partial^{-1}(f)(B_v(a_0, a_1, \cdots, a_n)) = \cdots\\
\cdots &=\partial^{-1}(f)\left(\sum_{i=0}^n (-1)^{in} (1 , a_i , \cdots , a_n , a_0 , \nu(a_1) , \cdots , \nu(a_{i-1}))\right) \\
&= \langle  \sum_{i=0}^n (-1)^{in}f([a_i | \cdots | a_n | a_0 | \nu(a_1) | \cdots | \nu(a_{i-1}) ]), 1 \rangle.
\end{aligned}$$
Combining this all, we get the required result.
\end{proof}
\subsection{Twisted $\bv$-operators}
Having built up this theory, we note that Theorem \ref{TwistedTensorTheorem} consists of the Hochschild Cohomology of ``twisted cochains'', $C^*(R, R_{\hat{b}})$. We should expect that these will give rise to ``twisted'' $\bv$-operators, which will in fact appear in our final result. In pursuit of this, we have the following result (which incidentally proves a fact we assumed earlier).
\begin{proposition}\label{frobprop} Given a Frobenius ring $R$, graded by the abelian group $A$, with graded Nakayama Automorphism $\nu$, admitting a twisting map $t: A \otimes B \rightarrow k^\times$ and $b\in B$, as above, we have
$$R_{\hat{b}} \cong D(R_{\nu \cdot \hat{b}^{-1}}).$$

\end{proposition}
\begin{proof}
The morphism realising this is as follows; 
$$\begin{aligned} R_{\hat{b}} &\rightarrow D(R_{\nu\hat{b}^{-1}}) \\
 r &\mapsto \langle r, - \rangle_R t(r,b)^{-1} \end{aligned}$$
We verify the actions agree. Take $x,r  \in R$, verifying the left action:
$$\begin{aligned} x\cdot r = xr &\mapsto \langle xr, - \rangle_R t(xr,b)^{-1}\\
&= \langle -, \nu(xr) \rangle_R t(xr,b)^{-1}\\
&= \langle r, -\nu(x) \rangle_R t(xr,b)^{-1} \\
&= \langle r, - \nu(x)t(x, b)^{-1} \rangle_R t(r,b)^{-1} \\
&= x \cdot \langle r, - \rangle_R t(r,b)^{-1} .\end{aligned}$$
and the right action:
$$\begin{aligned} r \cdot x = rx t(x,b) &\mapsto \langle rx, - \rangle_R t(r,b)^{-1} \\
&= \langle r, x-\rangle_R t(r,b)^{-1} \\
&=  \langle r, - \rangle_R t(r,b)^{-1} \cdot x .\end{aligned}$$
\end{proof} 
So then we have the following chain of isomorphisms
$$\begin{aligned}D(C_*(R,R_{\nu\hat{b}^{-1}}))&=\homo_k(R_{\nu\hat{b}^{-1}} \otimes_{R^e} \mathbbm{B}_\bullet (R), k) \\
& \cong \homo_{R^e}( \mathbbm{B}_\bullet (R), D(R_{\nu\hat{b}^{-1}})) \\
& \cong \homo_{R^e}( \mathbbm{B}_\bullet (R), R_{\hat{b}}) = C^*(R,R_{\hat{b}}),
\end{aligned}$$
whose composition we denote $\partial_b$. So in exact accordance with the untwisted version, we give the following definition
\begin{defn} Given $R$, $A$, $B$, as in the previous proposition and $b \in B$, we define the $\bv$-operator right twisted by $b$, $\Delta_b$, to be the morphism defined by
$$\begin{aligned}\Delta_b: \HH^*(R,R_{\hat{b}}) &\xrightarrow{\partial_b^{-1}} D(\HH_*(R,R_{\nu\hat{b}^{-1}})) \cong D(\HH_*^1(R,R_{\nu\hat{b}^{-1}})) \xrightarrow{D(B_{\nu\hat{b}^{-1}})} \cdots \\
\cdots &\rightarrow D(\HH_{*-1}^1(R,R_{\nu\hat{b}^{-1}}) \cong D(\HH_*(R,R_{\nu\hat{b}^{-1}})) \xrightarrow{\partial_b} \HH^{*-1}(R,R_{\hat{b}}).
\end{aligned}$$
 \end{defn}
\begin{lemma} Given $R$, $\nu$, $b \in B$ as above, we have that
$$\begin{aligned} \langle \Delta_b(f)([r_1, \ldots, r_n]), r_0 \rangle t(r_0,b) = \sum (-1)^{in}\langle f([r_i | \ldots |r_n | r_0 | \nu\hat{b}^{-1}(r_1) | \ldots | \nu\hat{b}^{-1}(r_i)], 1) \rangle .\end{aligned}$$
\end{lemma}
\begin{proof}
Recall our isomorphism $R \rightarrow D(R_{\nu\hat{b}^{-1}})$ is
$$r \mapsto \langle r , - \rangle t^{-1}(a,b).$$
Given $g \in C^n(R,R_{\hat{b}})$, 
$$\partial_b^{-1}(g)(r_0, \ldots, r_n) = \langle g([r_1 | \ldots | r_n]), r_0 \rangle t(g([r_1 | \ldots | r_n]),b)^{-1},$$
so as we still have the formula $\partial_b^{-1}(\Delta(f)) = D(B_{\nu\hat{b}^{-1}}) \circ \partial_b^{-1}(f)$
$$\begin{aligned}&\langle\Delta_b(f)  ([r_1 | \ldots | r_n]), r_0 \rangle t(\Delta_b(f)([r_1 | \ldots | r_n], b)^{-1} \\ 
=& \sum_{i=0}^n (-1)^{in} \langle f([r_{i} | \ldots |r_0 |\nu b^{-1}(r_1) | \ldots | \nu b^{-1}(r_{i-1})], 1 \rangle \ldots \\
 &  \ldots t(f([r_{i} | \ldots |r_0 |\nu b^{-1}(r_1) | \ldots | \nu b^{-1}(r_{i-1})]), b)^{-1}.
\end{aligned}$$
Now, we note that
$$|f([r_{i} | \ldots |r_0 |\nu b^{-1}(r_1) | \ldots | \nu b^{-1}(r_{i-1})])| = |f| + |[r_{i} | \ldots |r_0 |\nu b^{-1}(r_1) | \ldots | \nu b^{-1}(r_{i-1})]|,$$
where as $\nu$ is graded and shifting our elements does not alter grading, this is exactly
$$\ldots = |f| + |[r_0 | \ldots | r_n]| = |f| + |r_0| + |[r_1 | \ldots | r_n]|.$$
Further, we see that our inner product will vanish exactly when 
$$|f| + |r_0| + |[r_1 | \ldots | r_n]| + \sigma_R \not= 0.$$
But considering our original inner product, this demonstrates that $\Delta_b(f)$ is also graded, with degree $|\Delta_b(f)|=|f|$. So, 
$$|\Delta_b(f)([r_1 | \ldots | r_n])| = |f|+|[r_1 | \ldots | r_n]|,$$
which tells us that the difference between the twisting coefficients on both sides is $t(r_0, b)$. This gives us the formula as stated.
\end{proof}
We have the following corollary of Proposition \ref{frobprop}.
\begin{corollary} Given a Frobenius ring $S$, graded by an abelian group $B$, with graded Nakayama Automorphism $\nu$, admitting a twisting map $t: A \otimes B \rightarrow k^\times$, and $a \in A$, we have
$$\prescript{}{\hat{a}}{S} \cong D(S_{\nu \hat{a}}).$$
\end{corollary}
\begin{proof} We can compose the isomorphism from Proposition \ref{frobprop} with the isomorphism
$$\begin{aligned} \prescript{}{\hat{a}}{S} &\xrightarrow{\cong} S_{\hat{a}^{-1}} \\
 s &\mapsto t(a,s) s \end{aligned}$$
to get our result. Most crucially, the composition is given exactly by
$$s \mapsto \langle s, - \rangle$$
\end{proof}
We can trace out the same series of isomorphisms, to obtain $\prescript{}{a}{\partial}: D(\HH_*(S, S_{\nu \hat{a}})) \rightarrow \HH^*(S,\prescript{}{\hat{a}}{S}) $, which gives rise to the following definition.

\begin{defn} Given $S$, $A$, $B$, as in the previous proposition and $a \in A $, we define the $\bv$-operator left twisted by $a$, $\prescript{}{a}{\Delta}$ to be the morphism defined by
$$\begin{aligned}\prescript{}{a}{\Delta}: \HH^*(S,\prescript{}{\hat{a}}{S}) &\xrightarrow{\prescript{}{a}{\partial}^{-1}} D(\HH_*(S, S_{\nu \hat{a}}))  \cong D(\HH_*^1(S, S_{\nu \hat{a}}))  \xrightarrow{D(B_{\nu\hat{a}})} \cdots \\
\cdots &\rightarrow D(\HH_*^1(S, S_{\nu \hat{a}}))  \cong D(\HH_{*-1}(S, S_{\nu \hat{a}}))  \xrightarrow{\partial_b}  \HH^{*-1}(S,\prescript{}{\hat{a}}{S})\end{aligned}$$
And similarly, we have the formula
$$\langle \prescript{}{a}{\Delta}(g)([s_1 | \ldots | s_m]), s_0 \rangle = \sum_{i=0}^m (-1)^{im}\langle g([s_i | \ldots | s_0 | \nu\hat{a}(s_1) | \ldots | \nu\hat{a}(s_{i-1})], 1 \rangle$$
\end{defn}

\subsection{The Nakayama Automorphism and Cochains}\label{NakayamaCochains}
We will briefly consider the implications of some of these results. Earlier, we saw that given a Frobenius algebra $A$ and a semisimple automorphism $\nu:A \rightarrow A$ the map $C^1_*(A,A_\nu) \hookrightarrow C_*(A,A_\nu)$ is a quasi-isomorphism. But as $D(-)$ is exact, this tells us that the restriction
$$C^*(A,A) \rightarrow D(C^1_*(A,A_\nu))$$
is also a quasi-isomorphism. Given our map from earlier, $T:C_*(A,A_\nu) \rightarrow C_*(A,A_\nu)$, the map $T^*:=D(T)$ will act on $C^*(A,A)\cong D(C_*(A,A_\nu))$. In this way, we can consider the invariant subspace of $T^*$, which will be a subcomplex of $C^*(A,A)$, and we denote $C^*_\nu(A,A)$. Given $f \in C^*_\nu(A,A)$, we know that $T^*(f) =f$. Precisely, this says
$$ \partial^{-1}(f)(a_0,\ldots,a_n) = \partial^{-1}(f)(\nu(a_0), \ldots, \nu(a_n))$$
which, spelled out, means
$$\begin{aligned}  
\langle a_0 , f([\nu(a_1), \ldots, \nu(a_n)]) \rangle =&\langle f([\nu(a_1), \ldots, \nu(a_n)]), \nu(a_0) \rangle \\
=& \partial^{-1}(f)(\nu(a_0), \ldots, \nu(a_n)) \\
=&  \partial^{-1}(f)(a_0, a_1, \ldots, a_n) = \langle f([a_1, \ldots, a_n]), a_0 \rangle. \end{aligned}$$
In the case where $\nu$ is the Nakayama Automorphism of $A$, as $a_0$ was arbitrary here, this tells us that $ f([\nu(a_1), \ldots, \nu(a_n)]) = \nu(f([a_1,\ldots, a_n]))$. So then the $T^*$-invariant subspace of $C^*(A,A)$ are those functions which in some sense are $\nu$-equivariant. We then have the following morphisms of chain complexes:
$$C^*_\nu(A,A) \rightarrow C^*(A,A) \rightarrow D(C_*^1(A,A_\nu))$$
which levelwise we may write as 
$$D(C_*(A,A_\nu))^1 \rightarrow D(C_*(A,A)) \rightarrow D(C_*^1(A,A_\nu))$$
where we have again used $(-)^1$ to denote the $1$-eigenspace of $T^*$. We have the following lemma:
\begin{lemma} Given a finite dimensional vector space $V$ and a diagonalisable map $T:V \rightarrow V$ with non-empty $1$-eigenspace, the composition
$$i:D(V)^1 \rightarrow D(V) \rightarrow D(V^1),$$
is an isomorphism.\end{lemma}
\begin{proof} First we note that
$$D(V)^1 = \{ \delta \in D(V): \delta(Tv-v) =0 \textrm{ for all } v \in V \} = \operatorname{Im}(T-I)^0,$$
and further from basic linear algebra we know that 
$$\begin{aligned}\operatorname{Dim}(\operatorname{Im}(T-I)^0) &= \operatorname{Dim}(V) - \operatorname{Dim}(\operatorname{Im}(T-I)) \\
&=\operatorname{Dim}(\operatorname{Ker}(T-I)) = \operatorname{Dim}(D(V^1)).\end{aligned}$$
Thus the dimensions of these vector spaces are equal, and we need only show that our composition is an injection. We note that for $\delta \in D(V)^1$, $i(\delta)=0$ exactly when $\delta$ vanishes on $\operatorname{Ker}(T-I)$. So then, it suffices to show that $\operatorname{Im}(T-I)\oplus \operatorname{Ker}(T-I)=V$, for this will ensure that $\delta=0$. Now, that $T$ is diagonalisable means
$$V = \bigoplus_{\lambda \in k} \operatorname{Ker}(T-\lambda I).$$
We also see that $\operatorname{Im}(T-I) \subset \bigoplus_{\lambda \in k \backslash \{0\}} \operatorname{Ker}(T-\lambda I)$. But on each $\operatorname{Ker}(T-\lambda I)$, $T-I$ acts by $(\lambda-1)I$, so that $\operatorname{Ker}(T-\lambda I) \subset \operatorname{Im}(T- I)$. So
$$ \bigoplus_{\lambda \in k \backslash \{0\}} \operatorname{Ker}(T-\lambda I) = \operatorname{Im}(T-I),$$
and the direct sum is verified. 
\end{proof}
\begin{corollary} Suppose $A$ is a Frobenius algebra with semisimple automorphism $\nu:A\rightarrow A$. The inclusion
$$C^*_\nu(A,A) \rightarrow C^*(A,A)$$
is a quasi-isomorphism. \end{corollary}
\begin{proof} 
We begin by reducing to the diagonalisable case. Because $\nu: A \rightarrow A$ is semisimple, writing $\bar{A} := A \otimes_k \bar{k}$, the morphism $\bar{\nu}: \bar{A} \rightarrow \bar{A}$ is diagonalisable. Moreover we recall that
$$C^*(\bar{A}, \bar{A}) \cong C^*(A,A) \otimes_k \bar{k},$$
and we see that the map $\bar{T}:C^*(\bar{A}, \bar{A}) \rightarrow C^*(\bar{A}, \bar{A})$ is given by $T \otimes \operatorname{Id}_{\bar{k}}$. We can see that $C^*_{\bar{\nu}}(\bar{A}, \bar{A}) = C^*_\nu(A,A) \otimes \bar{k}$. So then we have that the inclusion
$$C^*_\nu(A,A) \rightarrow C^*(A,A),$$
which when tensored with $\bar{k}$ gives
$$C^*_{\bar{\nu}}(\bar{A}, \bar{A}) \rightarrow C^*(\bar{A}, \bar{A}).$$
As $\bar{k}$ is flat over $k$, this inclusion will be a quasi-isomorphism if and only if the original inclusion was a quasi-isomorphism. So it suffices to check in the diagonalisable case. But the diagonalisable case is verified by our previous lemma.
\end{proof}
This applies to our twisted complexes. In particular, we can consider $C^*(R,R_{\hat{b}}) \cong D(C_*^1(R,R_{\nu \hat{b}^{-1}}))$. To characterise our invariant subspace, take $f \in C_{\nu\hat{b}^{-1}}^n(R,R_{\hat{b}})$. We have that

$$ \partial^{-1}(f)(a_0,\ldots,a_n) = \partial^{-1}(f)(\nu\hat{b}^{-1}(a_0), \ldots, \nu\hat{b}^{-1}(a_n)),$$
which spelled out, means
$$\begin{aligned}  &\langle f([\nu\hat{b}^{-1}(a_1), \ldots, \nu\hat{b}^{-1}(a_n)]), \nu\hat{b}^{-1}(a_0) \rangle t^{-1}(f([\nu\hat{b}^{-1}(a_1), \ldots, \nu\hat{b}^{-1}(a_n)]), b) \\
=& \partial^{-1}(f)(\nu\hat{b}^{-1}(a_0), \ldots, \nu\hat{b}^{-1}(a_n)) \\
=&  \partial^{-1}(f)(a_0, a_1, \ldots, a_n) = \langle f([a_1, \ldots, a_n]), a_0 \rangle t(f[a_1,\ldots,a_n],b)^{-1}.\end{aligned}.$$
So then, as $\nu$ and $\hat{b}$ both preserve grading, this reduces to
$$\langle a_0, f([\nu(a_1), \ldots, \nu(a_n)]) \rangle t(a_0\ldots a_n, b)^{-1} = \langle  f([a_1, \ldots, a_n]), a_0 \rangle.$$

\section{Twisted Comparison Maps}
To compute the $\bv$-operator, we will go directly through the Bar Resolution of $R \twisten S$. Traditionally, in the untwisted setting we have morphisms
$$\begin{aligned} \operatorname{AW}:  \mathbb{B}(R \otimes S)  &\rightarrow \mathbb{B}(R) \ten \mathbb{B}(S)\\
\operatorname{EZ}: \mathbb{B}(R) \ten \mathbb{B}(S) &\rightarrow \mathbb{B}(R \otimes S).
\end{aligned}$$
Fortunately, twisted versions of these were computed in \cite{Shepler}, that we will denote $\awt$ and $\ezt$ and which in our case are easy to write down explicitly. We will go through the entire construction as the details will be crucial later on.
\begin{remark} 
It is worth noting that Shepler and Witherspoon use a subtly different notion of twisted tensor product from ours; crucially theirs is more general so that their construction applies to our situation. Their twist is built from a \emph{twisting morphism} 
$$\tau: S \otimes R \rightarrow R \otimes S$$
satisfying properties that we will not spell our here. In our case, for $r \in R ,s \in S$ homogeneous, $\tau(s \otimes r) = t(r,s)r \otimes s$.
Further, we remark that we first construct $\operatorname{\ezt}$ and $\operatorname{\awt}$ on the unreduced Bar resolutions, and subsequently pass to the reduced setting by a choice of splitting. 
\end{remark}
To proceed we will need some definitions that are standard from the construction of the Eilenberg-Zilber map.
\begin{defn} Write $\mathfrak{S}_n$ for the symmetric group on $n$ symbols. Given $p,q \in \mathbb{N}$ such that $p+q=n$, a $(p,q)$-shuffle in $\mathfrak{S}_n$ is a permutation $\sigma$ so that
$$\sigma(1) <  \cdots < \sigma(p) \textrm{ and } \sigma(p+1) < \cdots < \sigma(p+q).$$
We write $\mathfrak{S}_{p,q}$ for the collection of $(p,q)$-shuffles in $\mathfrak{S}_n$. 
\end{defn}
The maps $\awt$ and $\ezt$ are defined by performing the untwisted maps $\operatorname{AW}, \operatorname{EZ}$ and resolving the twisting on an intermediary resolution, denoted $Y_n$, defined as follows,
$$Y_n = R^{\otimes (n+2)} \otimes S^{\otimes (n+2)},$$
where $R$ acts on the left by acting on the leftmost component, and on the right acts on the rightmost $R$-component, twisted by the $S^{\otimes (n+2)}$ factor. $S$ acts similarly. Here, we have in some sense ``untwisted'' the regular Bar Resolution $\mathbf{B}(R \twisten S)$. Our penalty for doing this comes in twisting prefactors. Writing $\bezt$ and $\bawt$ for the respective maps on the unreduced Bar Resolution, we define maps
\begin{center}
\begin{tikzcd}
\mathbf{B}(R) \twisten \mathbf{B}(S) \arrow[r, "\theta", shift left=1 ex]  \arrow[rr, bend left=30, "\operatorname{\bezt}"{below}, dotted]& Y_\bullet \arrow[r, "\rho^{-1}", shift left=1 ex] \arrow[l, "\phi", shift left =1 ex] & \mathbf{B}(R\twisten S) \arrow[l, "\rho", shift left =1 ex]  \arrow[ll, bend left=30, "\operatorname{\bawt}"{below}, dotted]
\end{tikzcd}
\end{center}
\begin{itemize}
\item $\rho$ is defined bilinearly:
$$\rho_n: (R \twisten S)^{n+2} \rightarrow R^{\otimes (n+2)} \otimes S^{\otimes (n+2)}$$ 
$$\begin{aligned}
&(1_{R \otimes S}) \otimes (r_1 \otimes s_1) \otimes \cdots \otimes (r_n \otimes s_n) \otimes (1_{R \otimes S})\mapsto \cdots \\
\cdots &\mapsto \prod_{1\leq i<j\leq n} t(r_j,s_i)(1_R \ten r_1 \ten \cdots \ten r_n \ten 1_R) \otimes (1_S \ten s_1 \ten \cdots \ten s_n \ten 1_S) .\end{aligned}$$
\item  $\phi$ is defined bilinearly:
$$ \phi_n: R^{\otimes (n+2)} \otimes S^{\otimes (n+2)} \rightarrow (\mathbf{B}(R) \otimes \mathbf{B}(S))_n$$
where $\phi_n = \sum_{p+q=n} (-1)^{pq} \phi_{q,p}$, with $\phi_{q,p}: R^{\otimes (n+2)} \otimes S^{\otimes (n+2)} \rightarrow \mathbf{B}(R)_q \otimes \mathbf{B}(S)_p$ defined
$$\begin{aligned} & (1_R \ten r_1 \ten \cdots \ten r_n \ten 1_R) \otimes (1_S \ten s_1 \ten \cdots \ten s_n \ten 1_S) \mapsto \cdots \\
 \cdots &\mapsto (r_1 \cdots r_p) \otimes r_{p+1} \otimes \cdots \otimes r_n \otimes 1_R \otimes 1_S \otimes s_1 \otimes \cdots \otimes s_p \otimes (s_{p+1} \cdots s_n).
\end{aligned}$$

 \item $\theta$ is defined bilinearly:
 $$ \theta: (\mathbf{B}(R) \otimes \mathbf{B}(S))_n \rightarrow R^{\otimes (n+2)} \otimes S^{\otimes (n+2)} $$
where $\theta_n = \sum_{p+q=n} (-1)^{pq} \theta_{p,q}$, with $\theta_{p,q}:\mathbf{B}(R)_p \otimes \mathbf{B}(S)_q \rightarrow  R^{\otimes (n+2)} \otimes S^{\otimes (n+2)}$ defined
$$ \begin{aligned}  1_R \otimes r_1 \otimes \cdots \otimes r_p \otimes &1_R \otimes 1_S \otimes s_{p+1} \otimes \cdots \otimes s_n \otimes 1_S \mapsto \cdots \\
\cdots  \mapsto \sum_{\sigma \in \mathfrak{S}_{p,q}} \sign(\sigma) &1_R \otimes \sigma\cdot(r_1 \otimes \cdots \otimes r_p \otimes \underbrace{1_R \otimes \cdots \otimes 1_R}_{q\textrm{ many}}) \otimes 1_R \otimes \cdots \\ \cdots\otimes &1_S \otimes \sigma \cdot( \underbrace{1_S \otimes \cdots \otimes 1_S}_{p\textrm{ many}}\otimes s_{p+1} \otimes \cdots \otimes s_n) \otimes 1_S .\end{aligned}$$
\end{itemize}
The maps unreduced Eilenberg-Zilber and Alexander-Whitney maps are defined as in the diagram. Now we only need to take our maps from the unreduced Bar Resolution to the reduced Bar Resolution. To do this, we fix sections $i_R: \overline{R} \rightarrow R$ $i_S: \overline{S} \rightarrow S$ of the surjections $R \rightarrow \overline{R}, S \rightarrow \overline{S}$. We extend these to the Bar Resolutions to obtain morphisms
$$\mathbb{B}(R) \xhookrightarrow{\inc_R} \mathbf{B}(R) \xrightarrow{\pr_R} \mathbb{B}(R)  \hspace{20pt} \mathbb{B}(S) \xhookrightarrow{\inc_S} \mathbf{B}(S) \xrightarrow{\pr_S} \mathbb{B}(S) $$

$$\mathbb{B}(R \twisten S) \xhookrightarrow{\inc_{R \twisten S}} \mathbf{B}(R \twisten S) \xrightarrow{\pr_{R \twisten S}} \mathbb{B}(R \twisten S) $$
$$\mathbb{B}(R) \twisten \mathbb{B}(S) \xhookrightarrow{\inc_{R} \twisten \inc_{S}}\mathbf{B}(R) \twisten \mathbf{B}(S)\xrightarrow{\pr_{R} \twisten \pr_S} \mathbb{B}(R) \twisten \mathbb{B}(S)$$
\begin{remark} While projection is simply taken levelwise, the splitting is induced by the Comparison Lemma, and so a priori we have little control over $\ezt$ and $\awt$. This is however handled by the key points we extract from the proof of the below Proposition. \end{remark}
\begin{definitionproposition}[Theorem 6.1 of \cite{Shepler}]\label{proofoftwist} With our maps set up as earlier, we define the twisted Eilenberg-Zilber and Alexander-Whitney maps as follows:
$$\awt := (\pr_{R} \twisten \pr_S) \cdot \bawt \cdot \inc_{R \twisten S}$$

$$\ezt := \pr_{R \twisten S} \cdot \bezt \cdot (\inc_{R} \twisten \inc_{S})$$
they satisfy $\awt \ezt = 1$.
\end{definitionproposition}
\begin{proof} We will not copy the proof in \cite{Shepler}, though we will highlight some key points that we will use later. \begin{enumerate} 
\item $\pr_{R \twisten S} \circ \inc_{R \twisten S} = \operatorname{Id}$ and $\pr_R \twisten \pr_S \circ \inc_R \twisten \inc_S = \operatorname{Id}$
\item $\pr_{R \twisten S}$ and $\pr_R \twisten \pr_S$ apply the projections $R \twisten S \rightarrow \overline{R \twisten S}$, $R \rightarrow \overline{R}$, and  $S \rightarrow \overline{S}$ tensor-coordinate wise on the unreduced Bar Resolution. 
\item The morphism $\pr_{R \twisten S} \bezt$ vanishes on $\operatorname{Ker}(\pr_R \twisten \pr_S) \subset \mathbf{B}(R) \twisten \mathbf{B}(S)$
\item The morphism $\pr_{R} \twisten \pr_{S} \bawt$ vanishes on $\operatorname{Ker}(\pr_{R \twisten S}) \subset \mathbf{B}(R \twisten S)$
\end{enumerate}
\end{proof}
\section{Proof of Theorem}
Having discussed the preliminaries thoroughly, we can move on to the proof of our main result. Suppose $R$ and $S$ are Frobenius $k$-algebras graded by Abelian groups $A$ and $B$ respectively, admitting a twisting map $t:A \otimes B \rightarrow k^\times$. Suppose further that there exist $\sigma_R \in A$ and $\sigma_S \in B$ so that the inner products $R \otimes R \rightarrow k[\sigma_R]$ and $S \otimes S \rightarrow k[\sigma_S]$ are graded, and that $\nu_R$ and $\nu_S$ are semisimple.
\begin{theorem} Given $f \in C^n(R, R_{\hat{c}})^d$ and $g \in C^m(S, \prescript{}{\hat{d}}S)^c$, the $\bv$-operator defined in \cite{Lambre} $\Delta:\HH^*(R \twisten S, R \twisten S) \rightarrow \HH^{*-1}(R \twisten S, R \twisten S)$ is described as follows:

$$\tilde{\Delta}(f\boxtimes g) =  (-1)^{mn}(\Delta_c(f) \boxtimes g +(-1)^{n} f \boxtimes \; \prescript{}{d}{\Delta(g)})$$
\end{theorem}
\hspace{10pt}
\begin{proof} 
\begin{pfparts}
\item[\namedlabel{pfpart:setup}{Setup}.]
Before diving head first into reasoning about the box product of functions, we will consider exactly how our arguments are transformed by the ``Twisted Connes Operator''. So we begin with $h\in C_t^{N}(R \twisten S, R \twisten S)^{a,b}$; as the $\bv$-operator is defined on the reduced Bar Resolution of $R \twisten S$, we will have to compute $\Delta(h)$ as follows:
\begin{center}
\begin{tikzcd}
  C_t^{n+m}(R \twisten S, R \twisten S) \arrow[d, "(\awt)^*"] \arrow[r, dotted] &  C_t^{n+m-1}(R \twisten S, R \twisten S) \\
 C^{n+m}(R \twisten S, R \twisten S) \arrow[r,"\Delta"] & C^{n+m-1}(R \twisten S, R \twisten S) \arrow[u, "(\ezt)^*"]
\end{tikzcd}
\end{center}
So then, writing $\tilde{\Delta}$ for the $\bv$-operator on the twisted complex, and choosing $N \in \mathbbm{N}$ and $n,m$ so that $n+m=N-1$ we have that
$$\begin{aligned} \tilde{\Delta}(h)([r_1 | \cdots | r_n] \ten \els) &= (\Delta \circ (\awt)^*)(h) (\ezt([r_1 | \cdots | r_n] \otimes \els)) \\
 &= \Delta(h \circ \awt)(\ezt ([r_1 | \cdots | r_n] \otimes \els)) \end{aligned}$$
 where $[r_1 | \cdots | r_n] \in B_n(R)$, $\els \in B_m(S)$. So we have
 $$\begin{aligned} &\langle  \tilde{\Delta}(h)([r_1 | \cdots | r_n] \otimes \els), r_0 \otimes s_0 \rangle \\
 =& \langle \Delta(h \circ \awt)(\ezt ([r_1 | \cdots | r_n] \otimes \els)), r_0 \otimes s_0,  \rangle  \\
 =& \partial^{-1}(h \circ \awt)( B(r_0 \otimes s_0, \ezt([r_1 | \cdots | r_n ] \ten \els)))
 \end{aligned}$$
 Thus we first need to understand 
 $$\awt \circ ( B(r_0 \otimes s_0, \ezt([r_1 | \cdots | r_n ] \ten \els))$$
 \vspace{3pt}
\item[\namedlabel{pfpart:formofmap}{\noindent The form of $ {B(r_0 \otimes s_0, \ezt([r_1 | \cdots | r_n ] \ten \els))}$. }]
We recall
$$\begin{aligned} &\ezt(1_R \ten \elr \ten 1_R \ten 1_S \ten \els \ten 1_S) \\
=& \pr_{R \twisten S}\bezt(\inc_R(1_R \otimes \elr \otimes 1_R \otimes 1_S \otimes \els \otimes 1_S)) \end{aligned}$$
Our first concern is that a priori, we have little control over $\inc_R(1_R \otimes \elr \otimes 1_R \otimes 1_S \otimes \els \otimes 1_S)$, as it is induced from the comparison lemma. However, we know that
$$\begin{aligned} \inc_R([r_1 | \cdots | r_n]) & \twisten \inc_S(\els) - \cdots  \\ \cdots  &([\inc_R(r_1) | \cdots | \inc_R(r_n)] \twisten [\inc_S(s_1) | \cdots | \inc_S(s_m)]) \in \operatorname{Ker}(\pr_R \twisten \pr_S) \end{aligned}$$
So by point 3 of Definition/Proposition \ref{proofoftwist}
$$\begin{aligned} \ezt(\elr &\ten \els) = \pr_{R \twisten S} \bezt ([\inc_R(r_1) | \cdots | \inc_R(r_n) ] \ten [ \inc_S(s_1) | \cdots | \inc_S(s_m)]) \\
 &= \sum_{\sigma \in \mathfrak{S}_{n,m}} (-1)^{nm}\left(\prod_{\substack{1\leq a \leq n \leq b  \leq n+m \\ \sigma(b) < \sigma(a)}} t(r_a, s_b)^{-1} \right) \sign(\sigma) \sigma([r_1 | \cdots | r_n] \ten \els).
\end{aligned}$$
Writing $\nu$ as before, and $\pi:C_*(R \twisten S, R \twisten S_\nu) \rightarrow C_{*}(R \twisten S, R \twisten S_\nu)$ for
$$(1, a_0, \ldots, a_{N}) \mapsto (1, a_1, \ldots, \nu(a_0))$$
we can write 
$$B(a_0, \ldots, a_{N}) = \sum_{i=0}^n (-1)^{(i+1)N} \pi^i(1, a_1, \ldots, a_0)$$
And crucially we have 
\begin{equation}\label{formofb} B(r_0 \otimes s_0, \ezt (\elr \ten \els) = \hspace{-16pt}\sum_{\substack{i=0, \ldots, n+m  \\ \sigma \in \mathfrak{S}_{n,m}}}\hspace{-12pt}\lambda_{n,m, i,  \sigma}\pi^i(1, \sigma (r_1, \ldots, r_n, s_1, \ldots, s_m), r_0 \otimes s_0)\end{equation}
$$\lambda_{n,m, i, \sigma} := (-1)^{(i+1)(n+m)+nm} \sign(\sigma) \prod_{\substack{1\leq a \leq n \leq b  \leq n+m \\ \sigma(b) < \sigma(a)}} t(r_a, s_b)^{-1} $$\\[3pt]

 \item[\namedlabel{pfpart:elemenetdiscussion}{A Brief Discussion on Elements.}]
Before moving on to analysing the action of $\awt$, we briefly take stock and clear up some small details. First, note that in equation \ref{formofb} we have written $r_i, s_j$ for $r_i \otimes 1_S$ and $1_R \otimes s_j$, so that our equation fits on a single line, and we will continue this slight abuse of notation for the rest of the proof. We have also chosen to drop the $\inc_R$, again for compactness. Second, we note that, owing to fact 4 in the proof of Definition/Proposition \ref{proofoftwist}, similar to our reasoning concerning $\ezt$, we have for $[r_1 \otimes s_1 | \ldots | r_N \otimes s_N ] \in B(R \twisten S)$, that 
$$i_{R \twisten S}([r_1 \otimes s_1 | \ldots | r_N \otimes s_N ]) - [i_{R \twisten S}(r_1 \otimes s_1) | \ldots | i_{R \twisten S}(r_N \otimes s_N)] \in \operatorname{Ker}(\pr_{R \twisten S})$$
and thus their evaluations under $\pr_{R} \twisten \pr_S \bawt$ are equal. This allows us to describe $\awt$ by evaluating $\pr_R \twisten \pr_S \bawt$ on $[i_{R \twisten S}(r_1 \otimes s_1) | \ldots | i_{R \twisten S}(r_N \otimes s_N)]$. \\
At this point we will take a moment to clarify where all of our elements reside. Strictly, the element $B(r_0 \otimes s_0, \ezt (\elr \ten \els)$ lives in $C_{N}(R \twisten S, R \twisten S_\nu)$. At first glance, $h \circ \awt$ evaluates elements of $\mathbb{B}(R \otimes S)$, but we recall under the dualizing morphism $\partial$ that for $(a_0, \ldots, a_{N}) \in C_{N}(R \twisten S, R \twisten S_\nu)$ we will have
$$ \partial^{-1}(h \circ \awt) (a_0 , \ldots , a_{N}) = \langle (h \circ \awt)(1 \otimes [a_1 | \ldots | a_{N}] \otimes 1) , a_0\rangle_{R \twisten S}$$
So in context of the prior section of this proof, we will be interested in applying $f \boxtimes g \circ \awt$ to the element
\begin{equation*}\sum_{\substack{i=0, \ldots, n+m  \\ \sigma \in \mathfrak{S}_{n,m}}}\hspace{-12pt}\lambda_{n,m, i,  \sigma}\hat{\pi}^i[\sigma (r_1 \otimes 1 |\ldots | r_n \otimes 1 | 1 \otimes  s_1 | \ldots | 1 \otimes s_m) | r_0 \otimes s_0]\end{equation*}
where $\hat{\pi}[a_1, \ldots, a_{N}] = [a_{2}, \ldots, a_{N}, \nu(a_1)]$. From here we can begin to analyse the kernel of $\awt$.

 \item[\namedlabel{pfpart:kernelofawt}{The Kernel of $\mathbf{\awt}$.}]
We now seek to clarify which pure tensors remain after the application of $\awt$. Crucially, most terms vanish leaving us only with twisted $\bv$-operators. Proceeding from our prior discussion, set $r_{n+1}=\ldots=r_{n+m}=1_R$, $s_{1} = \ldots = s_{n}=1_S$, and writing $\tau := (n+m, n+m-1, \ldots, 1, 0) \in \mathfrak{S}_{n+m+1}$, $\rho^{-1}$ sends
$$\begin{aligned}
& \hat{\pi}^i[\sigma (r_1 \otimes 1 |\ldots | r_n \otimes 1 | 1 \otimes  s_1 | \ldots | 1 \otimes s_m) | r_0 \otimes s_0]\\
\mapsto& \mu_{n,m,i,\sigma}1_R \otimes  r_{(\tau^i\sigma)^{-1}(1)} \otimes \cdots \otimes r_0 \otimes \nu(r_{(\tau^i\sigma)^{-1}(n+m+1-i)}) \otimes \cdots \otimes  \nu(r_{(\tau^i\sigma)^{-1}(n+m)}) \otimes 1_R \otimes \cdots \\
& \cdots 1_S \otimes  s_{(\tau^i\sigma)^{-1}(1)} \otimes \cdots \otimes s_0 \otimes \nu(s_{(\tau^i\sigma)^{-1}(n+m+1-i)}) \otimes \cdots \otimes  \nu(s_{(\tau^i\sigma)^{-1}(n+m)}) \otimes 1_S
\end{aligned}$$
where the scaling prefactor is defined $\mu_{n,m,i,\sigma}:=\prod_{\substack{1 \leq a \leq n < b \leq n+m \\ \tau^i\sigma(b) < \tau^i\sigma(a)}} t(r_a, s_b)$.\\[2pt]

The map $\awt$ is the composition of this map with $\phi_{n+m+1} = \sum (-1)^{l(n-l)}\phi_{n+m+1-l, l}$. We note crucially that owing to the form of elements arising from $\ezt$, we have at most $(n+1)$ ``coordinates'' of our tensor which are not $1$; therefore, any mapping 
$$\phi_{n+m+1-l,l}: R^{\otimes (n+2)} \otimes S^{\otimes (n+2)} \rightarrow R  \otimes \bar{R}^{n+m+1-l}\otimes R \otimes S \otimes \bar{S}^l \otimes S$$ 
will vanish when $n+m+1-l > n+1$, as should any of the first tensor coordinates be a unit, the element will vanish in $\overline{R}^{n+m+1-l}$. This tells us that $m \leq l$. Similarly, we have at most $m+1$ non unital elements of $S$ in our pure tensor, so that we must have $l \leq m+1$. So we only have to consider two terms in this sum: $l=m,m+1$. \\[10pt]
$l=m$: We know that when $l=m$, we map into $R  \otimes \bar{R}^{n+1}\otimes R \otimes S \otimes \bar{S}^m \otimes S$, and so we know our non-unital elements of $R$ must be exactly in the last $n+1$ $R$-tensor coordinates in $Y_n$, or else our element will vanish under application of $\phi_{n+1,m}$. Suppose $i \leq n$. This forces our shuffle permutation $\sigma$ to be  
$$\sigma(a)= \begin{cases}  a \textrm{ for } a \leq i \\ a + m \textrm{ for } i < a \leq n \\  a+i-m \textrm{ for } n<a \leq n+m \\ \end{cases}$$   \\
all other permutations will vanish under $\phi_{n+1, m}$. To make this explicit, before application of $\rho^{-1}$ our elements will be as follows:
$$\begin{aligned} &\hat{\pi}^i[r_1 \otimes 1 | \ldots | r_i \otimes 1| 1 \otimes s_{n+1} | \ldots | 1 \otimes s_{n+m} | r_{i+1} \otimes 1 | \ldots | r_{n} \otimes 1 | r_0 \otimes s_0] \\
&= [1 \otimes s_{n+1} | \ldots | 1 \otimes s_{n+m} | r_{i+1} \otimes 1 | \ldots | r_0 \otimes s_0| \nu(r_1 \otimes 1) | \ldots |\nu(r_i \otimes 1)] \end{aligned}$$
and are mapped under $\rho^{-1}$ to 
$$\begin{aligned} & 1_R \otimes \underbrace{1_R \otimes \ldots  \otimes 1_R}_{m-\textrm{many}} \otimes  (r_{i+1}) \otimes \ldots \otimes r_0 \otimes \nu(r_1 ) \otimes \ldots \otimes \nu(r_i ) \otimes 1_R \otimes \\
&\cdots \otimes 1_S \otimes s_1 \otimes \ldots \otimes s_m \otimes \underbrace{1_S \otimes \ldots \otimes 1_S}_{(n-i)-\textrm{many}} \otimes s_0 \otimes 1_S \otimes \ldots \otimes 1_S \otimes 1_S,
\end{aligned}$$
which then under $\phi_{n+1, m}$ finally maps to
$$1_R \otimes [r_{i+1} | \ldots | r_0 | \nu(r_1) | \ldots | \nu(r_i) ] \otimes 1_R \otimes 1_S \otimes \els \otimes s_0 \in \mathbb{B}(R)_{n+1} \otimes \mathbb{B}(S)_m.$$
with the scaling prefactor, we have $\mu\lambda:=(-1)^{(i+1)n + nm}t(r_0,s_0)^{-1}\prod_{a\leq i} t(r_a,s_0 s_{n+1}\ldots s_{n+m})$.\\[5pt]
Now suppose that $i>n$. Then, we can keep track of where the element $r_0$ (independent of shuffle) is sent under the action of $\hat{t}^i$. This is the $(n+m+1-i)^{th}\leq m^{th}$ tensor coordinate of $R$. But then, we need the $n+1$ elements non unital elements to be in the last $(n+1)$-tensor coordinates, so that this pure tensor must vanish under $\phi_{n+1,m}$. \\[5pt]
$l=m+1$: We can perform an entirely similar argument for $l=m+1$, though our shuffle permutation will be different. First suppose that $i < n$. We need that all of the $m+1$ non-unital elements of $S$ must lie in the first $(m+1)$-tensor coordinates of $S$, so that our element does not vanish under application of $\phi_{n,m+1}$. But notably, $s_0$ will be situated in the $(n+m+1-i)^{th} > m+1^{th}$-tensor coordinate. So this term will vanish under $\phi_{n,m+1}$. \\
So we suppose that $i \geq n$. In this situation, we see the only shuffle permutation $\sigma$ producing to a non-vanishing term will be
$$\sigma(a) = \begin{cases} a+i \textrm{ for } a \leq n \\ a-n \textrm{ for }  n < a \leq i\\ a \textrm{ for } i<a. \end{cases}$$
Again, making this explicit, this element will be of the form\footnote{This is not quite correct for when $i=n$, though it does not take much imagination to see what the extension is.}
$$\begin{aligned} &\hat{\pi}^i[1 \otimes s_{n+1} | \ldots | 1 \otimes s_{i}| r_1 \otimes 1 | \ldots | r_n \otimes 1 | 1 \otimes s_{i+1} | \ldots | 1 \otimes s_{n+m} | r_0 \otimes s_0] \\
&= [1 \otimes s_{i+1} | \ldots | r_0 \otimes s_0 | \nu(1 \otimes s_{n+1}) | \ldots | \nu(1 \otimes s_{i}) | \nu(r_1 \otimes 1) | \ldots | \nu(r_n \otimes 1)], \end{aligned}$$
which then under $\rho^{-1}$ maps to
$$\begin{aligned} & 1_R \otimes \underbrace{1_R \otimes \ldots \otimes 1_R}_{(n+m+1-i)-\textrm{many}} \otimes r_0 \otimes \ldots \otimes \nu(r_{1}) \otimes \ldots \otimes \nu(r_n) \otimes 1_R \otimes \\
&\cdots \otimes 1_S \otimes s_{i+1} \otimes \ldots \otimes s_0 \otimes \nu(s_{n+1}) \otimes \ldots \otimes \nu(s_{i}) \otimes 1_S \otimes \ldots \otimes 1_S \otimes 1_S,
\end{aligned}$$
which then under $\phi_{n,m+1}$ maps to
$$r_0 \otimes [\nu(r_1) | \ldots | \nu(r_n) ] \otimes 1_R \otimes 1_S \otimes [s_{i+1} | \ldots | s_m | s_0 | \nu(s_1) | \ldots | \nu(s_{i}) ] \otimes 1_S\in \mathbb{B}(R)_{n} \otimes \mathbb{B}(S)_{m+1},$$
with prefactor $$\begin{aligned}\mu\lambda& = (-1)^{m(j+1) +n(m+1)}t(r_0,s_0)^{-1}\prod_{a> i} t(r_0 \ldots r_n, s_a) \\
&= t(a,b)t(r_0,s_0)^{-1} (-1)^{m(j+1) +n(m+1)}\prod_{n < a \leq i}t(r_0, \ldots r_n, s_a)^{-1}\end{aligned},$$
where we have written $j+n=i$, $|[s_0 | s_{n+1} | \ldots | s_{n+m}]| =: b$, $b-|s_0|=b'$, and $|[r_0 | \ldots | r_n]| =: a$.
\item[\namedlabel{pfpart:bvoperator}{The Form of the $\operatorname{BV}$-operator.}]
So we have all of the pieces to put together the form of our operator. We can now write down explicitly the form of the element
$$\begin{aligned} 
z:=& \awt \circ B(r_0 \otimes s_0, \ezt (\elr \ten \els))= \ldots  \\
\ldots= &\sum_{i=0, \ldots, n} (-1)^{x_i}t(r_0,s_0)^{-1}\prod_{a\leq i} t(r_a, b) 1_R \otimes [r_{i+1} | \ldots | r_0 | \nu(r_1) | \ldots | \nu(r_i) ] \otimes \ldots \\
&\hspace{50pt}\ldots \otimes 1_R \otimes 1_S \otimes \els \otimes s_0 + \ldots \\
\ldots+&\sum_{j=0, \ldots, m} t(a,b)t(r_0,s_0)^{-1} (-1)^{y_j}\prod_{n < l \leq n+j} t(a, s_l)^{-1} r_0 \otimes [\nu(r_1) | \ldots | \nu(r_n) ]  \otimes \ldots\\ &\hspace{50pt}\ldots \otimes 1_R \otimes 1_S \otimes [s_{j+n+1} | \ldots | s_{n+m} | s_0 | \nu(s_{n+1}) | \ldots | \nu(s_{j+n}) ] \otimes 1_S
\end{aligned}$$
where we have written $x_i := (i+1)n + nm$, $y_j := m(j+1) +n(m+1)$. From here we can write down explicitly our evaluation of $f \boxtimes g$; recall that for $b \in B$ we write $\hat{b}$ for the ring automorphism of $R$ scaling the homogeneous parts of $R$ by $t(b,-)$. Write $\nu_{b,R} = \hat{b} \circ \nu_R$, $\nu_{a,S} = \hat{a} \circ \nu_S$, which we will refer to as twisted Nakayama Automorphisms. \\[5pt]
It is crucial to recall that the map $\nu$ above is the Nakayama automorphism for $R\twisten S$. Applying this to $r \in R$, we get $\nu(r \otimes 1) = t(|r|,\sigma_S)\nu_R(r)=\nu_{\sigma_S, R}(r)$. The twisting prefactors will be absorbed into our twisted Nakayama automorphisms, so that our twistings are correct to use our twisted $\bv$-operators. \\[5pt]
We briefly recall from \cite{BriggsWitherspoon} the sign convention on $f\boxtimes g$. We have, for $r=[r_1|\ldots|r_n]\in B_n(R)$ and $[s]=[s_1|\ldots|s_n] \in B_m(S)$
$$f \boxtimes g ([r] \otimes [s]) = (-1)^{mn}f([r]) \otimes g([s]).$$
We know that the Hochschild Cohomology of $R \twisten S$ is spanned by elements of the form $f \boxtimes g \in C^{N}_t (R \twisten S, R \twisten S)$, where $f \in C^{n'}_\nu(R, R_{\hat{c}})^d$, $g \in C^{m'}(S, \prescript{}{\hat{d}}{S})^c$ with $n'+m'=N$. Crucially, we note that our map will vanish when $n'\not=n,n+1$. So then, fixing $n',m'$, when $n'=n+1$, we have $m=m'$ and:
$$ \begin{aligned} &f \boxtimes g (z) = (-1)^mt(r_0, b')\sum_{i=0}^n (-1)^{(i+1)n} f([r_{i+1} | \ldots | r_0 | \nu_{b+ \sigma_S, R}(r_1) | \ldots | \nu_{b+ \sigma_S, R} (r_i)]) \otimes \ldots \\
&\hspace{50pt}\ldots \otimes g(\els) s_0 ,\end{aligned}$$
where here the factor $\prod_{a\leq i} t(r_a, b)$ has been absorbed into the twisted Nakayama Automorphism. When $n'=n$, we have $m'=m+1$ so that
$$ \begin{aligned} &\hspace{-15pt}f \boxtimes g (z) = t(r_0, b')\sum_{j=0}^m (-1)^{(j+1)m} r_0 f([\nu_{ b+\sigma_S}(r_1) | \ldots | \nu_{ b+\sigma_S}(r_n)]) \otimes \ldots \\
&\hspace{30pt}\ldots \otimes g([s_{j+n+1} | \ldots | s_0 | \nu_{-a- \sigma_R, S}(s_{n+1}) | \ldots | \nu_{-a- \sigma_R,S} (s_{j+n})]),\end{aligned} $$
where we have manipulated the factors in the following way. The factors of $t(r_1\ldots r_n, b)$ and $\prod_{n < l \leq i} t(a, s_l)^{-1}$ were absorbed into the Nakayama Automorphisms on $R$ and $S$ respectively. This leaves us with a factor of $t(r_0,b)'$.\\[5pt] 
So we can split this into two cases.
\begin{itemize} 
\item When $n'=n+1$ 
$$\begin{aligned} &\langle \tilde{\Delta}(f \boxtimes g)([r_1 | \cdots | r_n] \ten \els), r_0 \otimes s_0 \rangle_{R \twisten S}=  \ldots \\
\ldots = &(-1)^mt(r_0,b')\sum_{i=0}^n (-1)^{(i+1)n}\langle 1_{R}, f([r_{i+1} | \ldots | r_0 | \nu_{b+ \sigma_S, R}(r_1) | \ldots | \nu_{b+ \sigma_S, R} (r_i)])  \rangle_R \ldots \\ 
& \hspace{20pt}\ldots\langle 1_S, g( \els)s_0 \rangle_S \end{aligned} $$
Crucially, we notice that this vanishes when
$$\langle 1_S , g( \els)s_0 \rangle_S = \langle g(\els), s_0 \rangle = 0.$$
Thus, we know we can assume $|s_0| + |g(\els)| + \sigma_S =0_S$. But we also know that $g$ is of degree $c$, so that this tells us
$$|s_0| + |\els| + c + \sigma_S =0_S$$
So that $b + \sigma_S = -c$, and by our previous discussion we can write this
$$\begin{aligned} &\langle \tilde{\Delta}(f \boxtimes g)(\elr \ten \els),r_0 \otimes s_0 \rangle_{R \twisten S}=  \ldots \\
&\hspace{30pt}\ldots = (-1)^m t(r_0,b'+c) \langle  \Delta_{c}(f)(\elr), r_0 \rangle_R \langle g( \els)s_0,  1\rangle_S \\
&\hspace{30pt} \ldots =(-1)^m\langle \Delta_c(f)([r_1 | \ldots | r_n]) \otimes g(\els), r_0 \otimes s_0 \rangle \end{aligned} $$
where the twisting factor or $t(r_0,c)$ comes from the form of the twisted $\bv$-operator. So crucially, as $r_0$, $s_0$ can vary over all of $R$, $S$ we have when $n'=n+1$
$$\tilde{\Delta}(f \boxtimes g)= (-1)^{n'm'}\Delta_c(f) \boxtimes g$$
where the minus sign comes from the form of the box product.
\item When $n'=n$, write $j=i-n$. We have
$$\begin{aligned} &\langle \tilde{\Delta}(f \boxtimes g)([r_1 | \cdots | r_n] \ten \els), r_0 \otimes s_0 \rangle_{R \twisten S}= \ldots \\
&\hspace{60pt}\ldots =t(r_0,b')\sum_{j=0}^m (-1)^{(j+1)m} \langle 1_R, r_0 f([\nu_{b+\sigma_S}(r_1) | \ldots | \nu_{b+\sigma_S}(r_n)]) \rangle_R  \ldots \\
&\hspace{60pt}\ldots \langle 1_S, g([s_{j+1} | \ldots | s_0 | \nu_{-a-\sigma_R, s}(s_1) | \ldots | \nu_{-a-\sigma_R,S} (s_{j})]) \rangle_S \end{aligned}. $$
Now just as before, we see that for our expression to be non-vanishing we must have $|f| + a + \sigma_R = 0$. So $d = |f| = -a - \sigma_R$. Moreover, we can also deduce that $c+b+\sigma_R=0$, so that we can write our expression as follows:
$$\begin{aligned} &\langle \tilde{\Delta}(f \boxtimes g)(\elr \ten \els),r_0 \otimes s_0 \rangle_{R \twisten S}=  \ldots \\
&\hspace{30pt}\ldots = t(r_0,b')\langle  r_0, f(T_{\nu \hat{c}^{-1}}(\elr))  \rangle_R \langle \prescript{}{d}{\Delta}(g)(\els), s_0 \rangle_S \end{aligned}$$
Now, as we chose $f \in C^{n'}_\nu(R, R_{\hat{c}})^d$ which we always may, we have the following invariance property stated at the end of Section \label{NakayamaCochains}.
$$t(r_0,c)\langle f([r_1, \ldots, r_n], r_0 \rangle_R = \langle  r_0, f(T_{\nu \hat{c}^{-1}}(\elr))  \rangle_R$$
So proceeding with this fact,
$$\begin{aligned}
&\hspace{30pt}\ldots = t(r_0,b')t(r_0,c)\langle  f(\elr), r_0 \rangle_R \langle \prescript{}{d}{\Delta}(g)(\els), s_0 \rangle_S \\
&\hspace{30pt}\ldots = t(r_0,b')t(r_0,c)t(r_0, c+b')^{-1}\langle   f(\elr) \otimes \prescript{}{d}{\Delta}(g)(\els), r_0 \otimes s_0 \rangle \\
&\hspace{30pt}\ldots=  \langle   f(\elr) \otimes \prescript{}{d}{\Delta}(g)(\els), r_0 \otimes s_0 \rangle, \end{aligned}$$
Thus, when $n'=n$
$$\tilde{\Delta}(f\boxtimes g) = (-1)^{ n'm'-n'} f \boxtimes \prescript{}{d}{\Delta(g)}$$
 \end{itemize}
 So finally, we recover our result
 $$\tilde{\Delta}(f\boxtimes g) =  (-1)^{m'n'}(\Delta_c(f) \boxtimes g +(-1)^{n'} f \boxtimes \prescript{}{d}{\Delta(g)})$$
\end{pfparts}
\end{proof}
\section{Computation of the $\bv$-structure on a family of Quantum Complete Intersections.}

As a demonstration of our methods, we will calculate the $\bv$-structure of quantum complete intersections $\Lambda_q(m,n)$, where $q$ is not a root of unity in $k$, just as in \cite{BriggsWitherspoon}. Calculating the $\bv$-operator when $q$ is a root of unity is nearly exactly the same, though there are more cases to consider.
\begin{remark} We note that, by definition of a $\bv$-algebra, we have for $f, g \in \HH^*(R \twisten S, R \twisten S)$ 
$$\Delta(f\cup g) -\Delta(f) \smile g + (-1)^{|f|} f\smile \Delta(g) = (-1)^{(|f|-1)|g|+1}[f,g]$$
so that to understand the $\bv$-operator it suffices to understand the $\bv$-operator on generators, as well as the cup product and the Gerstenhaber bracket. 
\end{remark}
\begin{proposition} Given $q\in k^\times$ not a root of unity, and $m, n\in \mathbbm{N}$ we have that
$$\HH^*(\Lambda_q(m,n)) \cong k[U]/U^2 \times_k \bigwedge\nolimits^*_k(V,W)$$
where $k[U]/U^2$ is in degree $0$, and $V,W$ in degree 1. The $\bv$-operator is given on generators by
$$\Delta(U)=0, \hspace{10pt}\Delta(V) =n-1, \hspace{10pt}\Delta(W) =m-1 $$
\end{proposition}
\begin{proof}
 We have the following projective resolution of $\Lambda(n)$,
\begin{center}
\begin{tikzcd}[column sep = huge]
P_\bullet: \ldots \arrow[r, "x \otimes 1 - 1 \otimes x"] & \Lambda(n)^{ev}[-n] \arrow[r, "\sum_{i=0}^{n-1} x^{n-1-i} \otimes x^i"] & \Lambda(n)^{ev}[-1] \arrow[r, "x \otimes 1 - 1 \otimes x"] & \Lambda(n)^{ev}
\end{tikzcd}
\end{center}
which after the application of $\operatorname{Hom}(-,\Lambda(n)_{\hat{b}})$ becomes
\begin{center}
\begin{tikzcd}[column sep = huge]
\operatorname{Hom}(P_\bullet, \Lambda(m)): \ldots &\arrow[l, "x (1-q^b)"]  \Lambda(n)[n]& \arrow[l, "x^{n-1}\sum_{i=0}^{n-1} q^{ib}"]  \Lambda(n)[1] &\arrow[l, "x (1-q^b)"]  \Lambda(n)
\end{tikzcd}
\end{center}
and taking cohomology, as in \cite{BriggsWitherspoon} we get
$$\HH^i(\Lambda(n), \Lambda(n)_{\hat{b}}) = \begin{cases}[@{}l@{\quad}r@{}l@{}] \Lambda(n) & & \textrm{if } i=0, b=0 \\ (x^{m-1}) & &\textrm{if } i=0, b\not=0 \\ \Lambda(n)/(nx^{n-1})[\frac{i}{2}n] & & \textrm{if } i>0 \textrm{ even, } b=0 \\ \operatorname{Ann}_{\Lambda(n)}(mx^{m-1})[\frac{i-1}{2}n+1] & & \textrm{if } i>0 \;\textrm{odd}, b =0 \\ 0 & & \textrm{otherwise} \end{cases}$$
Applying the theorem describing Hochschild Cohomology of twisted tensor products, we are left with 
$$\begin{aligned}
\HH^0(\Lambda_q(m,n))^{0,0} &= k(1 \otimes 1) \\
\HH^0(\Lambda_q(m,n))^{m-1,n-1} &= k(x^{m-1} \otimes x^{n-1}) \\
\HH^1(\Lambda_q(m,n))^{0,0} &= k(x[1] \otimes 1) + k(1 \otimes x[1]) \\
\HH^2(\Lambda_q(m,n))^{0,0} &= k(x[1] \otimes x[1])
\end{aligned}$$	
Now, writing $U:= x^{m-1} \otimes x^{n-1}$, $V= x[1] \otimes 1$ and $W= 1\otimes x[1]$, we first note that $\Delta(U)=0$ for degree reasons. From our theorem, we recover:
$$\Delta(x[1]\otimes 1) = \Delta_0(x[1]) \boxtimes 1 + (-1)^\alpha x[1] \boxtimes \Delta_0(1) =\Delta_0(x[1]) \boxtimes 1$$
But crucially, $\Delta_0(x)$ is simply the regular $\bv$-operator on $\HH^*(\Lambda(n),\Lambda(n))$. To compute this we need comparison maps between $P_\bullet$ and $\mathbf{B}_\bullet(\Lambda(n))$, but more precisely, for a comparison map $\psi:P_\bullet \rightarrow \mathbf{B}_\bullet(\Lambda(n))$, we need $\psi_1$. Such maps can be calculated inductively using weak self homotopies of $P_\bullet$, $\mathbf{B}_\bullet(R)$; doing this we learn that $\psi_1:P_1 \rightarrow B_1$ can be chosen to be
$$1 \otimes x^i \otimes 1 \mapsto \hspace{-10pt}\sum_{r+s=i-1} \hspace{-10pt} x^ry^s$$
We begin with $x[1] \in \HH^1(\Lambda(m),\Lambda(m))$, which is represented by the cochain $f:P_1 \rightarrow \Lambda(m)$ sending $1$ to $x$. Using our above constructed comparison map the representative $\tilde{f}:B_1 \rightarrow \Lambda(m)$ is the following composition:
$$ 1\otimes x^i \otimes 1 \xmapsto{\psi	_1} \hspace{-10pt}\sum_{r+s=i-1}\hspace{-10pt} x^ry^s \xmapsto{f} ix^i.$$
So then, we may compute the $\bv$-operator in the normal way, noting that $\Lambda(m)$ is symmetric 
$$\langle x^i, \Delta(\tilde{f})(1 \otimes 1) \rangle = \langle 1, f(1 \otimes x^i \otimes 1) \rangle = \langle 1, ix^i \rangle$$
from which we may deduce $\Delta(\tilde{f})(1 \otimes 1) = m-1$. The map $P_0 \rightarrow B_0$ is the identity, so that $\Delta(V)=m-1$. The process for computing $\Delta(W)$ is exactly the same.
\end{proof}
\begin{remark} We note that our calculation is consistent with those from \cite{BriggsWitherspoon}, wherein it is calculated that
$$[V,U] = (m-1)U$$
But our calculation tells us
$$[V,U] = -\Delta(VU) + \Delta(V) \smile U - V \smile \Delta(U) = \Delta(V) \smile U = (m-1)U$$
Note that this result also agrees with \cite{hou}.
\end{remark}

\bibliographystyle{plain} 
\bibliography{TwistedTensorReferences} 
\end{document}